\documentclass[12pt,a4paper]{amsart}
\usepackage[totalwidth=15.75cm,totalheight=22.275cm]{geometry}
\usepackage{amsmath,amsfonts,amsthm,amssymb,verbatim,enumerate,quotes,graphicx,mathtools}
\usepackage{color}

\newcommand{\comments}[1]{}
\newtheorem{theorem}{Theorem}

\newtheorem{definition}{Definition}

\newtheorem{lemma}{Lemma}
\newtheorem{corollary}{Corollary}

\newtheorem*{remark}{Remark}

\newtheoremstyle{claimstyle}%
   {}
   {}
   {\normalfont}
   {}
   {\itshape}
   {.}
   { }
   {\thmnote{#3}}
\theoremstyle{claimstyle}
\newtheorem*{varclaim}{}

\def\R{\mathbb{R}}
\def\C{\mathbb{C}}
\def\N{\mathbb{N}}
\def\No{\mathbb{N}_0}
\def\Z{\mathbb{Z}}
\def\Q{\mathbb{Q}}
\def\D{\mathbb{D}}

\def \s {\underline{s}}

\newcommand{\tef}{transcendental entire function}

\newcommand\qfor{\quad\text{for }}

\def\B{\mathcal{B}}
\def\Blog{\mathcal{B}_{log}}
\def\Blogn{\mathcal{B}^n_{log}}
\def\S{\mathcal{S}}

\numberwithin{equation}{section}
\numberwithin{theorem}{section}
\numberwithin{lemma}{section}
\numberwithin{definition}{section}
\makeatletter
\def\blfootnote{\xdef\@thefnmark{}\@footnotetext}
\makeatother
\begin{document}
%
%
%
%
\title[Dynamics in the Eremenko-Lyubich class]{Dynamics in the Eremenko-Lyubich class}
\author{David J. Sixsmith}
\address{Dept. of Mathematical Sciences \\
   University of Liverpool \\
   Liverpool L69 7ZL, UK \\
	\\ ORCiD: 0000-0002-3543-6969}
	\email{david.sixsmith@open.ac.uk}
%
%
%
%
\begin{abstract}
The study of the dynamics of polynomials is now a major field of research, with many important and elegant results. The study of entire functions that are not polynomials -- in other words transcendental entire functions -- is somewhat less advanced, in part due to certain technical differences compared to polynomial or rational dynamics. 

In this paper we survey the dynamics of functions in the Eremenko-Lyubich class, $\B$. Among transcendental entire functions, those in this class have properties that make their dynamics most ``similar'' to the dynamics of a polynomial, and so particularly accessible to detailed study. Many authors have worked in this field, and the dynamics of class $\B$ functions is now particularly well-understood and well-developed. There are many striking and unexpected results. Several powerful tools and techniques have been developed to help progress this work. There is also an increasing expectation that learning new results in transcendental dynamics will lead to a better understanding of the polynomial and rational cases.

We consider the fundamentals of this field, review some of the most important results, techniques and ideas, and give stepping-stones to deeper inquiry.
%
%
%
\end{abstract}
\maketitle

\parskip=0.45em
\tableofcontents
\parskip=0em

%
%
%
%
\section{Introduction}
\subsection{Goal and motivation}
The study of the dynamics of analytic functions is classical. In some sense, the work can be dated back to Euler \cite{Euler}, who studied ``towers'' of iterates of the form $z, z^z, z^{z^z}, \ldots$. The theory of the iteration of rational maps dates back to the memoirs of Fatou \cite{MR1504787,MR1504792,MR1504797} and Julia \cite{julia1918memoire}, in the first part of the twentieth century. Fatou, later, also considered transcendental maps \cite{MR1555220}, noting that these functions are not amenable to some of the techniques that can be successfully applied to rational maps. The latter part of the same century saw considerably renewed interest in the field, partly as a result of the fascinating and insight-giving computer graphics that became available, and partly as a result of powerful new techniques and ideas, such as Sullivan's use of quasiconformal maps in his proof of the ``No Wandering Domains Theorem'' \cite{MR819553}. Polynomial dynamics, in particular, is now a very deep and well-established field; even in popular culture there is an awareness of objects such as the Mandelbrot set. There are many remarkable results in polynomial dynamics, as well as challenging open questions, and many leading mathematicians are active in this area.

The study of transcendental dynamics, as a general field, is less developed, despite its comparable age. This is as a result of two factors that make advances in transcendental dynamics more difficult than those in polynomial dynamics. The first is that for a polynomial, $P$ say, the behaviour of $P$ at points of sufficiently large modulus is easy to describe; if $|z|$ is sufficiently large, then $|P(z)| > 2|z|$, and so the sequence $z, P(z), P(P(z)), \ldots$ tends to infinity uniformly, in the spherical metric, outside a sufficiently large disc. It readily follows from this that all the ``interesting'' dynamics of $P$ happens in a compact set. Clearly the Great Picard theorem implies that this behaviour is impossible for a {\tef}.

The second reason is that in dynamics it is frequently useful to study the branches of the inverse of the function. For a polynomial there are only finitely many points where some inverse branch does not exist, and much progress can be made by considering the properties of this finite set. For a {\tef} the situation can be much more complicated, as we discuss in Section~\ref{s.sing} below. Nonetheless, we will see that there is a large class of {\tef}s for which a similar property holds.

Despite these problems, the study of the dynamics of transcendental functions is very important, for several reasons. Firstly, there are already many beautiful and striking results. For example, in transcendental dynamics a set can naturally occur which consists of a collection of disjoint lines, and with the property that the set of endpoints of the lines has Hausdorff dimension $2$, but the set of lines without endpoints has Hausdorff dimension only $1$. (This result is known as Karpi{\'n}ska's paradox; see \cite{MR1696203}, and also \cite{MR2286634} where the set of lines fills the whole plane). Secondly, as noted by Rempe-Gillen \cite{lassearclike}, recent years have seen an increasing influence of phenomena from transcendental dynamics on the fields of rational and polynomial dynamics; it is not unreasonable to expect that a better understanding of the transcendental case will lead to further insights in the polynomial and rational settings. Fourthly, transcendental entire functions can exhibit features of dynamics which are impossible for polynomials or rational map; we study some of these, such as the arc-like continua of Theorem~\ref{theo:lassearclikeexample} below, in this survey. Finally, the dynamics of transcendental entire functions has applications in physics and other sciences.

In view of the difficulties with transcendental dynamics, noted above, but also the many strong reasons to seek to develop our understanding of this field, it seems natural to seek a class of {\tef}s that have properties that make them dynamically more ``similar'' to polynomials. For reasons we discuss below, one natural class of {\tef}s with this property is the so-called Eremenko-Lyubich class, which is usually denoted by $\B$. Our goal in this survey is to give a background to the dynamics of {\tef}s in this class, discuss some of the tools and techniques in this study, and present some of the notable results that are already in place.

This survey originated in a series of talks given at the school ``Topics in complex dynamics'' held in the Universitat de Barcelona in 2017.  It will be of interest to new researchers in the subject, who want to obtain an overview of the area before moving on to the somewhat technical papers proving the main results, to researchers from other areas of mathematics, who are interested in a short introduction to the subject, and to specialists in the study of holomorphic dynamics, who may need a guide to the fast-growing literature in this area. 

We assume that the reader is already somewhat familiar with complex analysis. 
\subsection{Overview of this survey}
We begin in Section~\ref{s.intro} by introducing some of the key terms and ideas from complex dynamics, along with some examples from within the Eremenko-Lyubich class. Readers already familiar with this field may skip this section entirely.

After this section, we motivate the setting for the class $\B$ by discussing singular values in Section~\ref{s.sing}, and then by explaining the definition of the class, and its implications, in Section~\ref{s.classes}. In Section~\ref{s.transform} we discuss in detail the \emph{logarithmic transform}, which is a key tool in working with functions in this class. We then, in Section~\ref{s.results}, use the logarithmic transform to prove some notable general results for these functions.

Next we discuss some particular sub-classes of the class $\B$, and show that many important questions in transcendental dynamics can be answered in some detail when restricted to these classes. First, in Section~\ref{s.disjoint}, we focus on maps of \emph{disjoint type}, for which there is -- in a sense -- a complete (topological) classification of the Julia set. We discuss this classification and some of the results that lie behind it. 

We then turn to Eremenko's conjecture. This conjecture, from the late 1980s, has motivated much work in transcendental dynamics, and is still open. We see that a strong version of the conjecture has been shown to be false even in the class $\B$. On the other hand, we show that the strong version of the conjecture does hold for class $\B$ maps that are of \emph{finite order}. For class $\B$ maps that are both of finite order and disjoint type, we obtain a particularly detailed description of the Julia set. This is all discussed in Section~\ref{s.finorder}.

This leads us to consider hyperbolic maps, in Section~\ref{s.hyperbolic}. In any dynamical setting, the hyperbolic maps are, in some sense, the maps that should be studied first. In fact, we show that the only {\tef}s which can be considered hyperbolic lie in the class $\B$. We study the Fatou and Julia sets of maps in this class. In Section~\ref{s.others} we review other noteworthy subclasses of the class $\B$.

Finally, in Section~\ref{s.bishop}, we review different techniques for constructing maps in the class $\B$, and how these constructions have been used to give examples of maps with novel dynamical behaviour. Much of this discussion focuses on the powerful recent techniques introduced by Bishop.
\subsection{Notation}
If $z \in \C$ and $r > 0$, then we denote the open disc, with centre $z$ and radius $r$, by $B(z, r)$. We let $\D$ denote the unit disc $\D := B(0,1)$. For simplicity it is useful to define $\No := \N \cup \{0\}$.   
\section{A brief introduction to complex dynamics}
\label{s.intro}
Our goal in this section is to give a brief overview of some ideas from complex dynamics, for the reader unfamiliar with this field. Useful resources on complex dynamics are the books \cite{beardon}, \cite{MR1230383}, and \cite{MR2193309}. Useful general references on transcendental dynamics are the surveys \cite{MR1216719} and \cite{MR2648691}.

Suppose that $U \subset \C$ is a domain, and $f : U \to U$ is an analytic function. In complex dynamics we study the \emph{iterates} of $f$, defined inductively by $f^1 := f$, and \[ f^{n+1} := f \circ f^n, \qfor n \in\N. \] In this section we will assume that $U$ is the whole complex plane; in other words that $f$ is entire.  In the rest of the paper we will assume that $f$ is transcendental; in other words, not a polynomial. If $z \in \C$, then we call the set of points $\{ z, f(z), f^2(z), \ldots \}$ the \emph{orbit} of $z$.

If $z \in \C$, and there is a smallest natural number $p$ such that $f^p(z) = z$, then $z$ is called a \emph{periodic point} of \emph{order} $p$. If $p=1$, then $z$ is called a \emph{fixed point}. The iteration of points near a periodic point is determined by the \emph{multiplier}, which is equal to $\lambda_z := (f^p)'(z)$. A periodic point is called \emph{repelling}, \emph{attracting} or \emph{indifferent} when $|\lambda_z| > 1$, $|\lambda_z| < 1$ or $|\lambda_z| = 1$ respectively. 

The \emph{Fatou} set, denoted by $F(f)$, can be defined as the set of points $z \in \C$ such that the family of iterates $(f^n)_{n\in\N}$ is equicontinuous with respect to the spherical metric in a neighbourhood of $z$. In other words, at points in the Fatou set, all sufficiently small perturbations of the starting point $z$ result only in small perturbations to $f^n(z)$ (on the Riemann sphere), for all values of $n$. So the Fatou set can, roughly speaking, be characterised as the set of points at which the iterates are not ``chaotic''. The \emph{Julia} set, denoted by $J(f)$, is the complement in $\mathbb{C}$ of $F(f)$, and so is the set of points at which the iterates are ``chaotic''. It is straightforward to see that all attracting periodic points lie in the Fatou set, and all repelling periodic points lie in the Julia set \cite[Theorems 6.3.1 and Theorem 6.4.1]{beardon}. The situation for indifferent periodic points is more complicated \cite[Sections 6.5 and 6.6]{beardon}, and we omit the detail.

By the Arzel\`{a}-Ascoli theorem, the Fatou set is equivalent to the set of points in a neighbourhood of which the family of iterates is a normal family. This is somewhat less intuitive than defining the set in terms of equicontinuity. However, this does imply that results such as Montel's theorem can be used to show that a point lies in the Fatou set.

It is easy to see from the definition that the Fatou set is open, and so (when non-empty) the Fatou set consists of one or more domains known as \emph{Fatou components}. These components can be classified as follows. Suppose that $U$ is a Fatou component. If, for some minimal $p \in \N$, $f^p(U) \cap U \ne \emptyset$, then $U$ is called \emph{periodic of order $p$}. (In fact in this case $U \setminus f^p(U)$ contains at most one point \cite{MR1642181}.) If $U$ is not periodic, but for some $q \in \N$, $f^{q}(U)$ meets a periodic Fatou component, then $U$ is called \emph{preperiodic}. Otherwise $U$ is called a \emph{wandering domain}; we discuss the question of the existence of wandering domains in the class $\B$ in Subsection~\ref{subs:bishop} below.

There is a well-known classification of periodic Fatou components \cite[Theorem 6]{MR1216719}. Roughly speaking, this states that if $U$ is a periodic Fatou component of a {\tef} $f$, of period $p$, then exactly one of the following four possibilities occurs:
\begin{itemize}
\item There is an attracting periodic point $z_0 \in U$, of period $p$, such that all points of $U$ tend to $z_0$ on iteration of $f^p$; such a $U$ is called an \emph{immediate attracting basin}. Points that eventually iterate into $U$ we call the \emph{attracting basin} of $z_0$.
\item There is an indifferent periodic point $z_0 \in \partial U$, of period $p$, such that all points of $U$ tend to $z_0$ on iteration of $f^p$; such a $U$ is called an \emph{immediate parabolic basin}. Points that eventually iterate into $U$ we call the \emph{parabolic basin} of $z_0$.
\item The function $f^p$ is univalent on $U$, and $f^p$ acts as a rotation on $U$ around an indifferent periodic point; such a $U$ is called a \emph{Siegel disc}.
\item All points of $U$ iterate to infinity; such a $U$ is called a \emph{Baker domain}. 
\end{itemize}
We will see in Theorem~\ref{noescaping} that no function in the class $\B$ can have a Baker domain. We give examples of the other three possibilities; the interested reader can check these functions are in the class $\B$ once they have read Section~\ref{s.classes} below. 
\begin{itemize}
\item Choose $a \in (1, \infty)$ and let $f_1(z) := e^z - a$. It can be checked that $f_1$ has two real fixed points, $p<0<q$ say, and that $p$ is attracting. So $f_1$ has an immediate attracting basin $U$. (In fact $U = F(f)$; see Theorem~\ref{theo:disjointFatou} below.) 
\item Consider $f_2(z) := e^z - 1$. Then $f_2$ has an indifferent fixed point at the origin. We can calculate that points of small modulus close to the negative real axis iterate towards $0$, whereas points of small modulus close to the positive real axis iterate away from the origin. It can then be shown that $0 \in J(f_2)$, and $f_2$ has an immediate parabolic basin. 
\item Set $f_3(z) := \mu(e^z - 1)$, where $\mu := \exp(2 \pi i \theta)$ and $\theta := \frac{1+\sqrt{5}}{2}$. It can be shown that $0 \in F(f_3)$, and $f_3$ has a Siegel disc, which consists of points that (roughly speaking) rotate around the origin under iteration of $f_3$; see \cite{lassethesis}. 
\item Finally, it is worth noting that for $f_4(z) := e^z$, we have $J(f_4) = \C$; in other words, the exponential map itself is ``chaotic'' in the whole plane. This was first proved in \cite{MR627790}; see also \cite{MR3447747} for an elementary and conceptual proof of this result.
\end{itemize}

Elementary properties of the Fatou and Julia sets of a {\tef} are given in the following; see, for example, \cite[Lemma 1, Lemma 2, Lemma 3, Theorem 3 and Theorem 4]{MR1216719}. A set $S$ is \emph{completely invariant} if $z \in S$ if and only if $f(z) \in S$. We use these properties without comment.
\begin{theorem}
\label{FJ}
Suppose that $f$ is a {\tef}. Then:
\begin{enumerate}[(a)]
\item $F(f) = F(f^n)$ and $J(f) = J(f^n)$, for $n\geq 2$.
\item $F(f)$ and $J(f)$ are completely invariant.
\item $J(f)$ is perfect, is either the whole of $\C$ or has empty interior, and is the closure of the set of repelling periodic points of $f$.
\end{enumerate} 
\end{theorem} 
\section{The role of singular values}
\label{s.sing}
\subsection{Motivation}
Suppose that $f : U \to V$ is analytic, and that $z \in V$. An important question in dynamics is the following; if $\Delta \subset V$ is a sufficiently small neighbourhood of $z$, can we define all inverse branches of $f$ in $\Delta$? It is well known that points where this is \emph{not} possible are particularly significant with regard to dynamics. 

As noted earlier, for polynomials these points, together with their orbits, determine the general features of the global dynamics. This is a very strong property because, as we shall see, these points are the images of points where the derivative is zero, and there are necessarily only finitely many of these. 

The situation for {\tef}s is, in general, much more complicated. To give an extreme example, there is a {\tef} for which the set of points in a neighbourhood of which some inverse branch cannot be defined is, in fact, the whole complex plane; see \cite{MR1511920}.
\subsection{The classification of singular values}
We need to be a little more precise, and so use a definition of singularities originally due to Iversen \cite{Iversen}; see also, for example, \cite{MR1344897}. For simplicity, we suppose that $f : \C \to \C$ is entire. Pick a point $w \in \C$. Suppose that for each $r > 0$, we have a method of choosing a component $U(r)$ of $f^{-1}(B(w, r))$ in such a way that $0 < r' < r$ implies that $U(r') \subset U(r)$. We let $r \rightarrow 0$; in other words, we study the intersection $U := \bigcap_{r>0} U(r)$. There are now three possibilities.

\begin{enumerate}[(a)]
\item $U = \{z\}$, a singleton, and $f'(z) \ne 0$. It follows by the inverse function theorem that, for sufficiently small values of $r>0$, $f$ is a homeomorphism from $U(r)$ to $B(w, r)$; we have obtained a branch of inverse locally. In this case $z$ is called a \emph{regular point}.
\item $U = \{z\}$, a singleton, and $f'(z) = 0$. It then follows that there exist $d > 1$ and $r>0$ such that $f$ is a $d$-to-$1$ map from $U(r)$ to $B(w, r)$, and there is no local inverse branch. In this case $z$ is called a \emph{critical point} and $w$ is called a \emph{critical value}.
\item $U = \emptyset$. In this case we call $w$ a \emph{finite asymptotic value}. This case cannot arise for a polynomial, but it is possible for a {\tef} (consider the exponential function). 
In this case there is no possibility to define an inverse branch; for sufficiently small values of $r$, the preimage $U(r)$ contains either zero or infinitely many preimages of the point $w$.
\end{enumerate}

We write $CP(f)$ for the set of critical points of $f$, $CV(f)$ for the set of critical values, and $AV(f)$ for the set of finite asymptotic values. We then define the set of singular values by
\[
S(f) := \overline{CV(f) \cup AV(f)}.
\]

The significance of this definition is as follows. If $z \notin S(f)$, then there is a neighbourhood $U$ of $z$ on which \emph{every} inverse branch of $f$ can be defined. In other words, $S(f)$ is the smallest closed set $S$ such that $$f : \C \setminus f^{-1}(S) \to \C \setminus S$$ is a covering map. Here, if $U, V$ are subdomains of $\C$, then $f : U \to V$ is a \emph{covering map} if it is continuous, and for each $z \in V$, there is a neighbourhood $\Delta$ of $z$ such that $f^{-1}(\Delta)$ is a union of disjoint open sets, each of which is mapped homeomorphically by $f$ onto $\Delta$. The covering is \emph{universal} if $U$ is simply connected. Covering maps have certain strong properties; see, for example, \cite{MR1185074} for more detail.

\begin{remark}\normalfont
We stress that, roughly speaking, $w \in S(f)$ means that there is \emph{an} inverse branch that cannot be continued through $w$. This does not mean that there might not be \emph{some} inverse branches which can be continued locally. For example, the function $$f(z) := z e^z$$ has two preimages of a small neighbourhood of the origin. It is easy to see that an inverse branch exists to one preimage, but not to the other.
\end{remark}
\subsection{The postsingular set}
In dynamics, it is often useful to consider the \emph{postsingular} set, which is defined by
\[
P(f) := \overline{\bigcup_{n \geq 0} f^n(S(f))}.
\]

The significance of this definition is as follows. If $z \notin P(f)$, then there is a neighbourhood $U$ of $z$ with the property that all inverse branches can be defined for \emph{all} iterates of $f$.
\section{The classes $\S$ and $\B$}
\label{s.classes}
\subsection{Definitions}
We want to study those {\tef}s whose set of singular values is easy to deal with. The following definition, therefore, is natural. 
\begin{definition}
\label{def:s}
A {\tef} $f$ is in the Speiser class $\mathcal{S}$ if $S(f)$ is finite.
\end{definition}
The exponential map and the cosine map are readily seen to lie in this class. The Speiser class is an analogue of the class of polynomials; both classes of maps have only finitely many singular values. The key difference is that maps in the Speiser class have an essential singularity at infinity, whereas, as observed earlier, for a polynomial infinity is an attracting fixed point. The study of the class $\S$ (also known as functions of \emph{finite type}) is classical; systematic study of functions in this class was undertaken by Nevanlinna, Teichm\"uller and Speiser. The study of the dynamics of these functions was initiated in \cite{MR769199} and \cite{MR857196}.
 
For many purposes, however, Definition~\ref{def:s} is unnecessarily restrictive, and the following is more useful. 
\begin{definition}
A {\tef} $f$ is in the Eremenko-Lyubich class $\mathcal{B}$ if $S(f)$ is bounded.
\end{definition}
The class $\B$ (also known as functions of \emph{bounded type}) seems to have been first introduced by Eremenko and Lyubich \cite{MR1196102}, where they studied the dynamics of functions in this class. The focus of this survey is on functions in the class $\B$. It is clear that $\mathcal{S} \subset \mathcal{B}$. Moreover, this inclusion is strict; see, for example, the function $f(z) := z^{-1}\sin z$, discussed below. 

Note that both these classes of functions are closed under composition; this statement follows from the observation that if $f = g \circ h$, then
\begin{equation}
\label{compos}
CV(f) = g(CV(h)) \cup CV(g) \quad \text{and} \quad AV(f) = g(AV(h)) \cup AV(g).
\end{equation}

We should note that these classes are not small. For example, Bishop \cite{Bish3} gave a very general construction of {\tef}s with no finite asymptotic values, and critical values $\pm 1$; in other words, these functions all lie in the class $\S$. We discuss this construction in Section~\ref{s.bishop}.
\subsection{Examples}
We give some examples; the details in each case are left as an exercise.

\begin{itemize}
\item Let $f(z) := \lambda \exp z$, for some $\lambda \in \C\setminus \{0\}$. Then $$CV(f) = \emptyset, \ S(f) = AV(f) = \{0\}.$$ This family of functions, which lies in the class $\S$, is known as the exponential family. These are, in a sense, the entire maps with the simplest dynamics after the polynomials, and they have been very widely studied; see, for example, \cite{MR758892,MR1053806,MR627790,lassethesis,MR1956142}.
\item Let $f(z) := a e^z + b e^{-z}$, for $a, b \in \C\setminus\{0\}$. Then $$AV(f) = \emptyset, \ S(f) = CV(f) = \left\{\pm2\sqrt{ab}\right\}.$$ This family of functions, which also lies in the class $\S$, is known as the cosine family and has also been widely studied; see, for example, \cite{MR2458810, MR2286634}.
\item Let $f(z) := z^2 \exp(-z^2)$. Then $AV(f) = \{0\}$, $CV(f) = S(f) = \{ 0 , 1/e\}$. This function again is in class $\S$. Observe that a finite asymptotic value can also be a critical value.
\item Let $f(z) := z^{-1}\sin z$. Then $AV(f) = \{0\}$, and $CV(f)$ is an infinite set of real numbers, all of modulus not greater than one.  This is an example of a function in class $\B$ but not in class $\S$. See also \cite[p.286]{nevanlinna}
\item Let $f(z) := \int_0^z \exp(-e^t) \ dt$. Then $CV(f) = \emptyset$. However, see \cite{MR1642181}, there is an $\alpha \in \R$ (which can be calculated) such that $AV(f) = \{ \alpha + 2k\pi i : k \in \Z \}$. This function is not in class $\B$, but illustrates a function with infinitely many finite asymptotic values. It can be shown, using \eqref{compos}, that the function $$g(z) := \exp((f(z)-\alpha)^2)$$ is a class $\B$ function such that $AV(g)$ is an infinite set.
\end{itemize}
\section{The logarithmic transform}
\label{s.transform}
\subsection{Definition}
In this section we discuss the construction of the logarithmic transform, which is a central technique when working with dynamics in the class $\B$. The transform was first used in dynamics in \cite{MR1196102}, although a similar change of variable appears in the work of Teichm\"uller and also (as observed in \cite{MR2885574}) in the work of Speiser himself \cite[p.295]{MR1509398}. Note that we will make a number of statements which we will not fully justify; these are left as an exercise.

Suppose that $f \in \mathcal{B}$. Let $D$ be a Jordan domain that contains the singular values of $f$, as well as the set $\{0, f(0)\}$. Note, by \cite[Proposition 2.9]{MR3433280}, that $f^{-1}(D)$ is connected. Set $W := \C \setminus \overline{D}$ and $\mathcal{V} := f^{-1}(W)$. It follows that the components of $\mathcal{V}$, which are called \emph{tracts of $f$}, are Jordan domains the boundary of which passes through infinity. The function $f$ is a universal covering from each of these components to $W$. Note that this is already a useful property of functions in class $\B$.

Now set $H := \exp^{-1}(W)$ and $\mathcal{T} := \exp^{-1}(\mathcal{V})$; see Figure~\ref{f1}. (Note that the notation of Figure~\ref{f1} will be used extensively throughout this paper.) Each component of $\mathcal{T}$ is simply connected, with boundary homeomorphic to $\R$, and is known as a \emph{tract of $F$}. 

\begin{figure}
	\includegraphics[width=14cm,height=10cm]{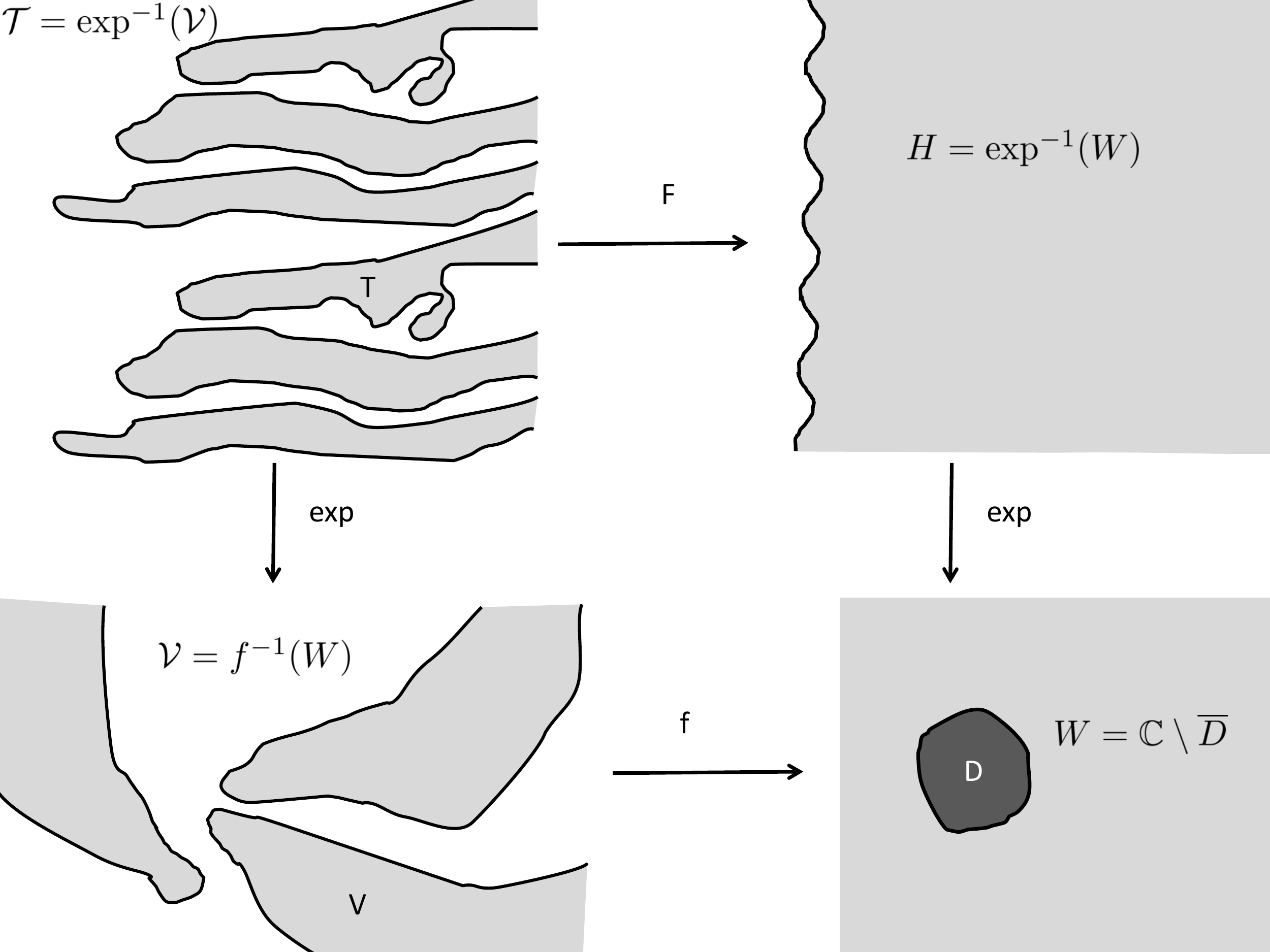}
  \caption{An illustration of the logarithmic transform}\label{f1}
\end{figure}

We can then \emph{lift} $f$ to a map $F : \mathcal{T} \to H$ satisfying $\exp \circ F = f \circ \exp$, and which can be chosen to be $2 \pi i$ periodic. The existence of this map follows from the fact that $f$ is a universal covering of $W$, which (on the Riemann sphere) is topologically a punctured disc; it is a well-known result of the theory of covering maps that such functions can be factored in this way \cite[Theorem 5.10]{MR1185074}. We call $F$ a \emph{logarithmic transform} of $f$; note that $F$ depends not only on $f$ but also on the choice of the domain $D$.
\subsection{Properties of a logarithmic transform}
\label{subs:properties}
We have the following properties:

\begin{enumerate}[(A)]
\item $H$ is a $2 \pi i$ periodic unbounded Jordan domain containing a right half-plane. \label{log1}
\item $\mathcal{T} \ne \emptyset$ is $2 \pi i$ periodic, and the real part is bounded from below in $\mathcal{T}$ but unbounded from above.
\item Each component $T$ of $\mathcal{T}$ is an unbounded Jordan domain disjoint from all its $2 \pi i$ translates.
\item For each such $T$, the map $F : T \to H$ is a conformal isomorphism. ($T$ is called a tract of $F$).\label{iso} 
\item It follows from the Carath\'eodory-Torhorst theorem \cite[Theorem 2.1]{MR1217706} that we may extend this function continuously to the closure of $\mathcal{T}$ in the Riemann sphere, and then we also have that $F(\infty) = \infty$.
\item The components of $\mathcal{T}$ accumulate only at infinity; in other words, if $(z_n)_{n\in\N}$ is a sequence of points of $\mathcal{T}$, each belonging to a different component of $\mathcal{T}$, then $z_n \rightarrow\infty$ as $n\rightarrow\infty$.\label{log6}
\end{enumerate}

One of the important reasons for using the transform is property \eqref{iso}. Although $f$ is just a covering map, the transform $F$ is conformal on each tract. In particular, this means that for each tract $T$, $F|_T$ has an inverse, which we denote by $$F^{-1}_T : H \to T.$$ We stress that we use this notation frequently in the following.

\subsection{The class $\Blog$}
We denote by $\Blog$ the class of those functions that satisfy properties \eqref{log1}--\eqref{log6}, irrespective of whether they are the transforms of a {\tef}. Many results regarding the dynamics of functions in class $\B$ are derived from results regarding the dynamics of functions in class $\Blog$, and so it is often appropriate to state them in this more general context.

Note that, using the Riemann mapping theorem, it is straightforward to create functions in class $\Blog$ simply by writing down definitions of subsets of the plane, $\mathcal{T}$ and $H$, that meet the geometric conditions above. It is a striking fact that, with the additional condition that $\overline{\mathcal{T}} \subset H$, we can ''recover'' a {\tef} $f$ with the ``same'' dynamics of $F$. We discuss this in further detail in Section~\ref{s.bishop}.
\subsection{Results from complex analysis}
\label{subs:hyperbolic}
Since we will use it twice, we note first the well-known Koebe quarter theorem; see, for example, \cite[Theorem 4.1]{MR2492498}.
\begin{theorem}
Suppose that $\phi : \D \to \C$ is conformal. Then $$\phi(\D) \supset B(\phi(0), |\phi'(0)|/4).$$
\end{theorem}
It is easy to see how this result can be modified for a conformal map from any disc.

%
We also need to introduce, at this point, a brief summary of the properties of the hyperbolic metric; the interested reader will find more detail in \cite{MR2492498}. First, in the disc $\D$ we define the hyperbolic density by
\[
\rho_\D(z) := \frac{2}{1-|z|^2}, \qfor z \in \D.
\]

Now suppose that $U$ is a simply connected proper subdomain of $\C$. By the Riemann mapping theorem, there is a conformal map $\phi : U \to \D$. We then define the hyperbolic density in $U$ by
\[
\rho_U(z) := \rho_\D(\phi(z)) |\phi'(z)|, \qfor z \in U.
\]
It can be shown, see \cite[p.25]{MR2492498}, that this definition is independent of the choice of $\phi$. 

We use the hyperbolic density to define the hyperbolic length of a piecewise smooth curve $\gamma \subset U$ by
\[
\ell_U(\gamma) := \int_\gamma \rho_U(z) \ |dz|,
\]
and then define the hyperbolic distance between two points $z_1, z_2 \in U$ by
\[
d_U(z_1, z_2) := \inf_{\gamma} \ell_U(\gamma),
\]
where the infimum is taken over all piecewise smooth curves $\gamma$ that join $z_1$ to $z_2$ in $U$.

The Schwarz-Pick Lemma \cite[Lemma 6.4]{MR2492498} states that if $f : U \to V$ is an analytic map between simply connected proper subdomains of $\C$, then
\[
d_V(f(z_1), f(z_2)) \leq d_U(z_1, z_2), \quad\text{and}\quad \rho_V(f(z_1)) |f'(z_1)| \leq \rho_U(z_1), \qfor z_1, z_2 \in U,
\]
with equality when $f$ is an analytic bijection, and strict inequality otherwise. In particular, if $f$ is conformal, then
\begin{equation}
\label{eq.conformal}
d_{f(U)}(f(z_1), f(z_2)) = d_U(z_1, z_2), \qfor z_1, z_2 \in U.
\end{equation}

If $U \subsetneq V$, then the Schwarz-Pick Lemma applied to the identity map gives that
\begin{equation}
\label{eq.SP}
\rho_V(z) < \rho_U(z), \qfor z \in U.
\end{equation}

The hyperbolic density at a point is closely related to its distance to the boundary. Indeed, see \cite[Theorems 8.2 and 8.6]{MR2492498}, we have the following \emph{standard estimate} on the hyperbolic density in a simply connected proper subdomain of $\C$; 
\begin{equation}
\label{eq.hypest}
\frac{1}{2\operatorname{dist}(z, \partial U)} \leq \rho_U(z) \leq \frac{2}{\operatorname{dist}(z, \partial U)}, \qfor z \in U.
\end{equation}
Here, for a point $z$ and a set $W$, we use dist$(z, W)$ to denote the Euclidean distance dist$(z, W) := \inf_{w \in W} |z-w|$. Note that the right-hand inequality is a consequence of the Schwarz-Pick Lemma, and the left-hand inequality is deduced using the Koebe quarter theorem.

Equations \eqref{eq.conformal} and \eqref{eq.hypest} have important implications for the mapping of points under $F \in \Blog$, which is a conformal map from each tract to $H$. Very roughly speaking, suppose that $z_1, z_2$ both lie in the same tract $T$. Since $T$ is disjoint from its $2 \pi i$ translates, any curve from $z_1$ to $z_2$ in $T$ must necessarily stay ``close'' to the boundary of $T$. However, if $F(z_1)$ and $F(z_2)$ have large real parts, then a curve from $F(z_1)$ to $F(z_2)$ can be chosen ``far'' from the boundary of $H$. It then follows from \eqref{eq.conformal} and \eqref{eq.hypest} that $|F(z_1) - F(z_2)|$ must be large compared to $|z_1 - z_2|$.
\section{Some important general results in the class $\B$}
\label{s.results}
\subsection{An expansion property}
The following simple property of functions in the class $\Blog$, given in \cite{MR1196102}, is fundamental, since it says that $F$ is \emph{expanding} at points where it is of sufficiently large real part.
\begin{lemma}
\label{lemm:expansion}
Suppose that $F : \mathcal{T} \to H$ is in the class $\Blog$. Suppose that $R \in \R$ is such that $\{ w \in \C : \operatorname{Re} w > R \} \subset H$. Then
\begin{equation}
\label{eq.expansion}
|F'(w)| \geq \frac{1}{4\pi}(\operatorname{Re} F(w) - R), \qfor w \in \mathcal{T} \text{ such that } \operatorname{Re} F(w) > R.
\end{equation}
\end{lemma}
\begin{proof}
Choose a component $T$ of $\mathcal{T}$. Suppose that $w \in T$ is such that $\operatorname{Re} F(w) > R$. Note that
\[
B(F(w), \operatorname{Re} F(w) - R) \subset H.
\]
Set $z= F(w)$ and $\rho = \operatorname{Re} z - R$. By applying the Koebe quarter theorem to $F_T^{-1}$ we deduce that $T$ contains a disc with centre $w$ and radius $\frac{1}{4}|(F_T^{-1})'(z)|\rho$. However, we know that $T$ cannot contain a disc of radius $\pi$. The result then follows by a calculation.
\end{proof}
\begin{remark}\normalfont
Suppose that $F$ is a logarithmic transform of a {\tef} $f$. Then \eqref{eq.expansion} says that the quantity $|zf'(z)/f(z)|$ tends to infinity as $f(z)$ tends to infinity. (This quantity is actually the derivative of $f$ in the cylinder metric.) In \cite{MR3233575} it was shown that this property, in fact, characterises the Eremenko-Lyubich class, in the sense that if $f \notin \B$, then 
\[
\liminf_{R\rightarrow\infty} \left\{ \left|z\frac{f'(z)}{f(z)}\right| : |f(z)| > R \right\} = 0.
\]
\end{remark}
\subsection{No escaping points in the Fatou set}
The escaping set of a {\tef} $f$ is defined by
\[
I(f) := \{ z \in \C : f^n(z) \rightarrow \infty \text{ as } n \rightarrow \infty\}.
\]
Eremenko \cite{MR1102727} showed that $J(f) \cap I(f) \ne \emptyset$ and that $J(f) = \partial I(f)$. There are many {\tef}s for which $F(f) \cap I(f) \ne \emptyset$; for example, it can be easily shown that for the function $f(z) := 1 + z + e^{-z}$, every point in the right half-plane lies in $F(f) \cap I(f)$. The following theorem, from \cite{MR1196102}, uses Lemma~\ref{lemm:expansion} to show that this is impossible in the class $\B$.
\begin{theorem}
\label{noescaping}
Suppose that $f \in \B$. Then $J(f) = \overline{I(f)}$.
\end{theorem}
\begin{proof}
We are required to prove that $F(f) \cap I(f) = \emptyset$. Suppose, by way of contradiction, that $z \in F(f) \cap I(f)$. Since the Fatou set is open, we can choose $r>0$ such that all iterates tend uniformly to infinity on the disc $\Delta := B(z, r)$. Replacing $z$ by a point $f^k(z)$, for some $k \geq 0$, if necessary, we can assume that $f^n(\Delta) \subset W$, for $n\geq 0$. (Recall that $W = \C \setminus \overline{D}$, where $D$ is a Jordan domain that contains the singular values of $f$.)

Let $C$ be a component of $\exp^{-1}(\Delta)$. Note that $\exp F^n(C) = f^n(\Delta)$, for $n \geq 0$, and so the real part of the iterates of $F$ tend uniformly to positive infinity on $C$. Moreover, $F^n(C) \subset \mathcal{T}$, for $n\geq 0$.

Choose any point $w \in C$, and let $\rho_n$ denote the radius of the largest disc, centre $F^n(w)$, contained in $\overline{F^n(C)}$. Since $F$ is univalent, it follows by the Koebe quarter theorem that $\rho_{n+1} \geq \frac{1}{4}\rho_n|F'(F^n(w))|$. Now, it follows, by \eqref{eq.expansion}, that $$|F'(F^n(w))|\rightarrow\infty \text{ as } n \rightarrow \infty.$$ Hence $F^n(C)$ contains arbitrarily large discs. However, $\mathcal{T}$ cannot contain a disc of radius greater than $\pi$. This contradiction completes the proof.
\end{proof}
\subsection{Normalized transforms}
Equation \eqref{eq.expansion} says, roughly, that we can make $|F'|$ arbitrarily large by making $\operatorname{Re} F$ large. This was easy to use in the proof of Theorem~\ref{noescaping}, because of the nature of points in the escaping set. However, it is often convenient to be able to assume that $|F'|$ is bounded away from one.

The established convention is as follows. If $F\in\Blog$ is such that $H = \{ z : \operatorname{Re} z > 0 \}$, and $|F'(z)| \geq 2$, for $z \in \mathcal{T}$, then we say that $F$ is \emph{normalized}. We denote the set of normalized functions by $\Blogn$.

Although we do not use them directly in this paper, many results in the literature apply only to functions in the class $\Blogn$. However, it is often no significant loss of generality to assume that a function in the class $\Blog$ is normalized. For, if $F$ is not normalised, then we can substitute with a normalized function as follows. First we choose $R>0$ sufficiently large that both $\tilde{H} := \{ z : \operatorname{Re} z > R \} \subset H$ and also $|F'(z)| \geq 2$, for $F(z) \in \tilde{H}$. We then consider instead the function defined by $\tilde{F}(z) := F(z + R) - R$, which maps $F^{-1}(\tilde{H}) - R$ to $\{ z : \operatorname{Re} z > 0 \}$. This is a normalised function that is conjugate to $F$, and so has the same dynamics. 
\subsection{The Hausdorff dimension of the Julia set}
Finally in this section, we discuss a proof which uses some of these ideas. Baker \cite{MR0402044} proved that the Fatou set of a {\tef} has no unbounded, multiply connected components. This implies that the Julia set of a {\tef} must contain a continuum, and so has Hausdorff dimension not less than one. In 1996, Stallard \cite{MR1357062} proved the important result that the Julia set of a function in the class $\B$ necessarily has dimension greater than one. We will outline the idea behind a more recent \cite{MR2480096} proof of this fact; we stress that the authors of  \cite{MR2480096}, in fact, proved a stronger result. Our goal here is primarily to illustrate how working with the logarithmic transform can lead to results in the original dynamical plane. 
\begin{theorem}
\label{theo:dimension}
If $f \in \B$, then $\dim_H J(f) > 1$.
\end{theorem}
\begin{proof}
Let $f \in \B$, let $F$ be a logarithmic transform of $f$, and choose a tract $T$ of $F$. Choose a large number $R>0$, and let $Q$ be a square of side $R$ centred at the point $R$ on the real line. We choose $R$ large enough that $Q$ intersects with many of the $2\pi i$ periodic copies of $T$.

It can be shown to follow from \eqref{eq.expansion} that if $R$ is chosen sufficiently large, then $T$ contains a preimage $V$ of $Q$, which lies in $H$; essentially this follows from the earlier discussions regarding the hyperbolic metric. Since $F$ is $2\pi i$ periodic, this gives rise to infinitely many preimages of $Q$, each of the form $V + 2\pi i k$, for $k \in \Z$, all lying in $H$.

Now we consider the preimages of these sets, lying in $T$. There are infinitely many of these, and considerations of hyperbolic geometry show that they are spaced along $T$. Hence, if $R$ is chosen large enough, $Q$ contains many components of $F^{-2}(Q)$, lying in $T$ and its $2 \pi i$ translates; see Figure~\ref{f2}. 

\begin{figure}
	\includegraphics[width=14cm,height=8cm]{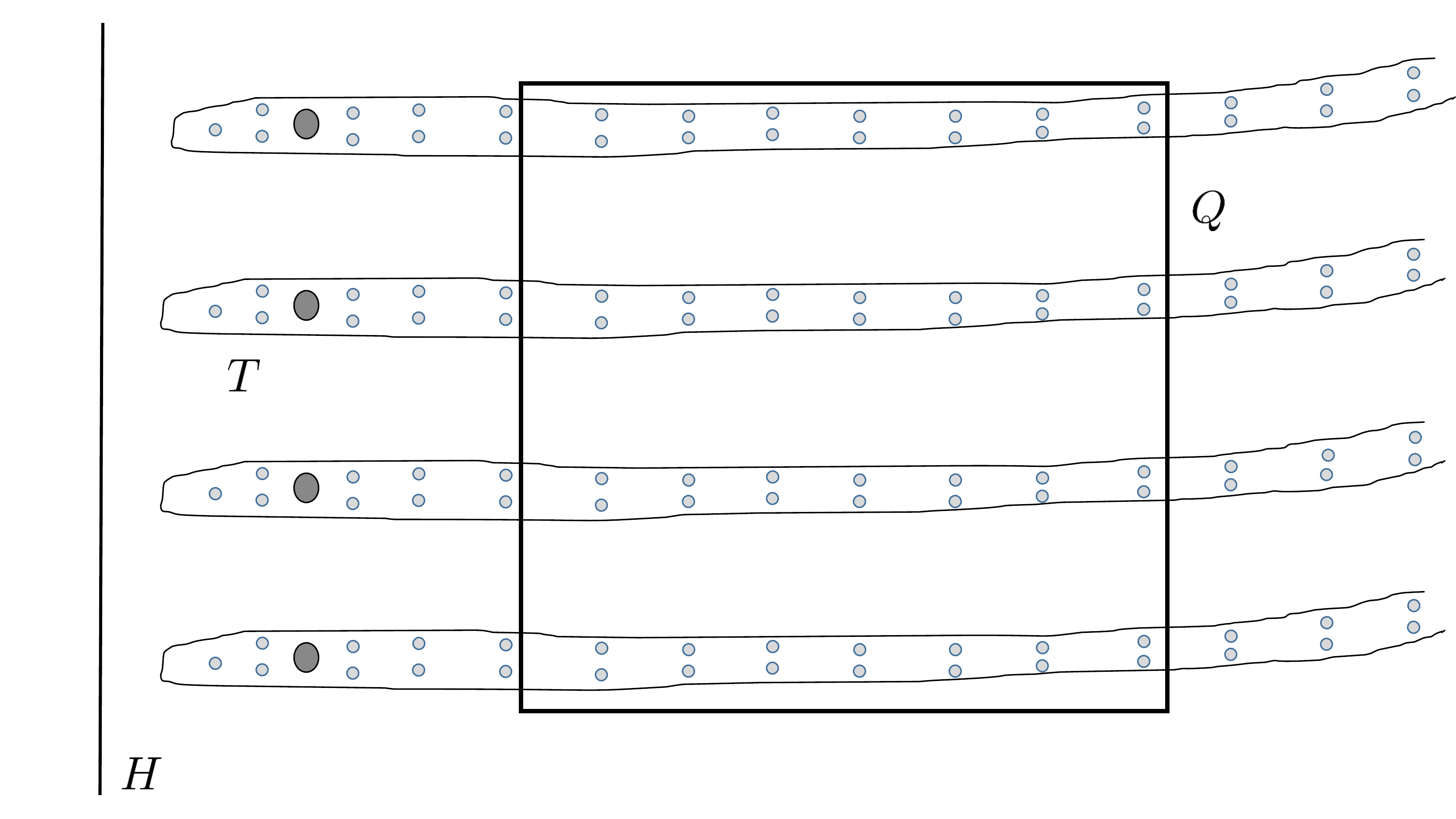}
  \caption{An illustration of the sets in the proof of Theorem~\ref{theo:dimension}. Components of $F^{-1}(Q)$ are shown in dark gray, and components of $F^{-2}(Q)$ in light gray.}\label{f2}
\end{figure}

Let $\Phi_1, \Phi_2, \ldots, \Phi_N$ denote the inverse branches of $F^{-2}$ that map $Q$ into $Q$. These maps define a \emph{conformal iterated function system}, and such systems are well-understood. In particular we can set $Q_0 := Q$, inductively define
\[
Q_{k} := \bigcup_{n=1}^N \Phi_n(Q_{k-1}), \qfor k \in \N,
\]
and then let
\[
X := \bigcap_{k \geq 0} Q_k.
\]
This is known as the \emph{limit set}, and consists of those points of $Q$ which remain in $Q$ under iteration of $F^2$. The integer $N$ and the derivatives $|(\Phi_n)'|$ can be estimated, and it can be shown from these estimates, using standard techniques, that the limit set $X$ has Hausdorff dimension greater than one whenever $R$ is sufficiently large. We omit the detail.

Finally we consider the set $X' = \exp X$. It is straightforward to see that the Hausdorff dimension of $X'$ is greater than one. Moreover $|(F^n)'|$ tends to infinity in $X$ and so $|(f^n)'|$ tends to infinity in $X'$. Since points in $X'$ do not iterate to infinity, we can deduce that $X' \subset J(f)$. This completes the proof.
\end{proof}
\subsection{The Julia set of the transform}
Notice that in this proof we considered points the orbit of which is contained in $\overline{\mathcal{T}}$. We then showed that the exponential of these points lies in the Julia set of $f$. This suggests the following definitions. First we define the Julia set of $F$ by
\[
J(F) := \{ z \in \C : F^n(z) \in \overline{\mathcal{T}} \text{ for } n \geq 0 \}.
\]
Notice that the Julia set is precisely those points on which all iterates are defined, and so there is no value, from a dynamical perspective, in defining a ``Fatou set'' as the complement of $J(F)$.

In view of Theorem~\ref{noescaping}, it also makes sense to define the escaping set of $F$ by
\[
I(F) := \{ z \in J(F) : \operatorname{Re} F^n(z) \rightarrow +\infty \text{ as } n \rightarrow \infty \}.
\]

The orbit of a point in $J(F)$ is characterised by the closures of the tracts it visits. Accordingly, we let $\mathcal{A}$ denote the set of tracts of $F$, and then for each $z \in J(F)$ there is a unique $\underline{s} = s_0 s_1 \ldots \in \mathcal{A}^{\No}$ such that $F^n(z) \in \overline{s_n}$, for $n \in \No$. In this case we call $\underline{s}$ the \emph{external address} of $z$. Note that if $\s$ is the external address of $z$, then $F(z)$ has external address $\sigma(\s)$, where $\sigma$ is the \emph{Bernoulli shift map} defined by $$\sigma(s_0 s_1 \ldots) := s_1 s_2 \ldots.$$

One might ask if it is the case that $J(f) = \exp J(F)$. In general this is not true; for example, the orbit of points in $J(f)$ may include points of small modulus that lie in $D$. However, for a normalised function it is the case that $\exp J(F) \subset J(f)$; the proof of this follows by a technique very similar to that used in the proof of Theorem~\ref{lemm:expansion}.

Finally, for each external address $\underline{s}$ we let
\[
J_{\underline{s}}(F) := \{ z \in J(F) : z \text{ has external address } \underline{s} \}.
\] 
It is possible for $J_{\underline{s}}$ to be empty. If $J_{\underline{s}}(F) \ne \emptyset$, then we say that $\underline{s}$ is \emph{admissible}. 
%
%
%

Suppose that $z_0 \in J(F)$. As observed earlier, $z_0$ has an external address say $\underline{s} = s_0 s_1 \ldots \in \mathcal{A}^{\No}$ such that $F^n(z_0) \in \overline{s_n}$, for $n \in \No$. Recall that we set $F(\infty) = \infty$. Set $S_0 := \overline{s_0} \cup \{\infty\}$, and then inductively
\begin{equation}
\label{eq:pullback}
S_k := \{ z \in S_{k-1} : F^k(z) \in \overline{s_k} \cup \{\infty\} \}.
\end{equation}

We can deduce that 
\begin{equation}
\label{eq:jcdef}
\hat{J}_{\underline{s}}(F) := J_{\underline{s}}(F) \cup \{\infty\} = \bigcap_{k\in\No} S_k,
\end{equation}
is a non-empty, closed, connected subset of the Riemann sphere that contains the points $z_0$ and $\infty$. In other words, $\hat{J}_{\underline{s}}(F)$ is a \emph{continuum};  we call it a \emph{Julia continuum} of $F$. Note that if $$\underline{s} = s_0 s_1 \ldots \ne \underline{s'} = s_0' s_1' \ldots,$$ then there is a $k \in \No$ such that $s_k \ne s'_k$. Hence $F^k(J_{\underline{s}}(F))$ and $F^k(J_{\underline{s'}}(F))$ lie in different tracts, and so $J_{\underline{s}}(F)$ and $J_{\underline{s'}}(F)$ are disjoint. Note that we cannot assume that $J_{\underline{s}}(F)$ is connected, although in the next section we will discuss a class of functions for which this is the case. 

\section{Functions of disjoint type}
\label{s.disjoint}
In this section we discuss a class of functions with the property that $J(f) = \exp J(F)$. This enables us to give particularly strong characterisations of the Fatou and Julia sets of these functions.
\subsection{Definition and examples}
We begin with a definition. Recall that $D$ is a Jordan domain that contains the singular values of $f$, as well as the set $\{0, f(0)\}$.
\begin{definition}
A function $f \in \B$ is of \emph{disjoint type} if $D$ can be chosen so that $f(\overline{D}) \subset D$.
\end{definition}
If $f$ is of disjoint type, then we assume that $D$ has been chosen so that $f(\overline{D}) \subset D$, in which case we also have that $\overline{\mathcal{V}} \subset W$; the fact that, therefore, $\partial \mathcal{V} \cap \partial W = \emptyset$ explains the choice of term ``disjoint type''. Note also that if $F$ is the transform of $f$, then $\overline{\mathcal{T}} \subset H$. Whenever this property holds for some $F \in \Blog$ (not necessarily the logarithmic transform of a {\tef}), then we say that $F$ is of disjoint type.

Examples of class $\B$ functions of disjoint type are those in the much-studied family 
\[
f_\lambda(z) := \lambda e^z, \quad\text{ where } 0 < \lambda < 1/e.
\]
It is easily seen that $S(f_\lambda) \cup \{0, f_\lambda(0)\} \subset \D$ and $f_\lambda(\overline{\D}) \subset \D$. It is also easy to show that $f_\lambda$ has an attracting fixed point $p_\lambda > 0$, and a repelling fixed point $q_\lambda > p_\lambda$. The Fatou set of $f_\lambda$ is the immediate attracting basin of $p_\lambda$. It can then be shown that the Julia set of $f_\lambda$ consists of an uncountable set of unbounded curves; see \cite[p.50]{MR758892}. We will see that these properties are characteristic of a function of disjoint type.

Maps of disjoint type are important for three reasons. Firstly, the properties of these maps allow us to build a particularly clear understanding of their dynamics. Secondly, any map $f \in \B$ has a disjoint type map in its parameter space (by which we mean the space of maps of the form $\lambda f$, for $\lambda \in \C\setminus\{0\})$. This is expressed in the following; see \cite[p.261]{MR2570071}.
\begin{lemma}
\label{lemm:smallerfun}
Suppose that $f\in\B$. Then there exists $\lambda > 0$, such that the map $$g_\lambda(z) := \lambda f(z),$$ is of disjoint type.
\end{lemma}
\begin{proof}
First choose a value of $\lambda$ sufficiently small to ensure that $g_\lambda(\overline{\D}) \subset \D$. Since we have that $S(g_\lambda) = \lambda S(f)$, we can decrease $\lambda$, if necessary, to obtain that $S(g_\lambda) \subset \D$. Then $g_\lambda$ is of disjoint type.
\end{proof}
Thirdly, it is often the case that for a given map $f \in \B$, we can transfer properties of the dynamics of the disjoint type map $g_\lambda(z)$ from Lemma~\ref{lemm:smallerfun} back to the dynamics of $f$; see, for example, \cite[Section 5]{MR2912445} and \cite[Section 5]{MR2570071}, and the discussion of hyperbolic maps in Section~\ref{s.hyperbolic} below.
\subsection{The Fatou set of a disjoint type map} 
The Fatou set of a disjoint type function is described in the following; see  \cite[Proposition 2.8]{MR2912445}. In fact (b) below is often used as an alternative definition of a disjoint type map.
\begin{theorem}
\label{theo:disjointFatou}
Suppose that $f \in \B$. Then the following are equivalent.
\begin{enumerate}[(a)]
\item The function $f$ is of disjoint type.
\item The Fatou set of $f$ is connected, and $P(f)$ is a compact subset of $F(f)$.
\item The function $f$ has a unique attracting fixed point and $P(f)$ is a compact subset of its immediate basin of attraction.
\end{enumerate}
\end{theorem}
\begin{proof}
We show first that (a) implies (b). It follows by Montel's theorem that $D$ is contained in a Fatou component, $U$ say, of $f$. It is easy to see that $U$ must be an immediate attracting basin. Note also that $S(f) \subset U$, and so $P(f) \subset U$. 

It is known \cite{MR1234740} that all finite limit functions in a wandering domain of a {\tef} lie in the postsingular set. We can deduce, using Theorem~\ref{noescaping}, that $f$ has no wandering domains. As noted earlier, the classification of periodic Fatou components \cite[Theorem 6]{MR1216719} gives that all periodic Fatou components of $f$ are immediate attracting basins, parabolic basin, Siegel discs, or Baker domains. Since Baker domains lie in $I(f)$, it follows again by Theorem~\ref{noescaping} that $f$ has no Baker domains.
 
It is known \cite[Theorem 7]{MR1216719} that any cycle of parabolic domains meets the singular set, and the boundary of a Siegel disc lies in the postsingular set. Since $P(f) \subset U$, we can deduce that $f$ has no parabolic domains or Siegel discs. In particular, $U$ is an attracting basin containing $S(f)$, and so $P(f)$ is a compact subset of $U$.

Since any cycle of immediate attracting basins contains at least one singular value \cite[Theorem 7]{MR1216719}, $U$ is the unique immediate attracting basin of $f$. Since the attracting fixed point of $f$ lies in $D$, we deduce that 
\[
F(f)= \bigcup_{n\in\No} f^{-n} (D).
\] 

Now
\[
S(f) \subset D \subset f^{-1}(D) \subset f^{-2}(D) \ldots.
\]
We can deduce from this, by \cite[Proposition 2.9]{MR3433280}, that each set $f^{-n}(D)$ is connected. Thus, $F(f)$ is an ascending union of connected sets, and so is connected.

Next we show that (b) implies (c). Since $F(f)$ is connected, there is a completely invariant Fatou component $U$ such that $U = F(f)$. Since $P(f)$ is a compact subset of $U$, it follows by an argument similar to the above that $U$ must be the immediate attracting basin of an attracting fixed point of $f$.

The fact that (c) implies (a) follows by properties of attracting domains, and is omitted.
\end{proof}
\subsection{Uniform hyperbolic expansion}
Recall from Lemma~\ref{lemm:expansion} that a map $F \in \Blog$ is expanding whenever $\operatorname{Re} F$ is sufficiently large. An important property of disjoint type maps is that they are \emph{uniformly} expanding on $\mathcal{T}$. This is expressed precisely in the following proposition \cite[Lemma 2.1]{MR2753600}. Here, if $W$ is a simply connected proper subdomain of $\C$, and $F$ is analytic in $W$, then we define the \emph{hyperbolic derivative} of $F$ in $W$ by 
\[
||DF(z)||_W := \frac{\rho_W(F(z))}{\rho_W(z)}|F'(z)|, \qfor z, F(z) \in W.
\]
Roughly speaking, the quantity $||DF(z)||_W$ determines how expanding (or contracting) $F$ is, at $z$, in the hyperbolic metric in $W$. 
\begin{lemma}
\label{lemm:unifexpansion}
Suppose that $F : \mathcal{T} \to H$ is a function in the class $\Blog$ of disjoint type. Then there is a constant $L > 1$ such that 
\[
||DF(z)||_H \geq L, \qfor z \in \mathcal{T}.
\]
Moreover, if $z_1, z_2$ belong to the same tract of $F$, then $d_H(F(z_1),F(z_2)) \geq  L d_H(z_1, z_2)$.
\end{lemma}
\begin{proof}
Since $F$ is a conformal map from each component of $\mathcal{T}$ to $H$, we have that
\[
||DF(z)||_H = \frac{\rho_H(F(z))}{\rho_H(z)}|F'(z)| = \frac{\rho_{\mathcal{T}}(z)}{\rho_H(z)}\frac{\rho_H(F(z))}{\rho_{\mathcal{T}}(z)}|F'(z)| = \frac{\rho_{\mathcal{T}}(z)}{\rho_H(z)}, \qfor z \in \mathcal{T}.
\]
Since $\mathcal{T} \subset H$, we deduce by \eqref{eq.SP} that $||DF(z)||_H > 1$, for $z \in \mathcal{T}$. In particular, if $K \subset \mathcal{T}$ is compact, then there is a $\lambda > 1$ such that $||DF(z)||_H \geq \lambda$, for $z \in K$. We need to show, therefore, that $||DF(z)||_H$ does not tend to one as $z$ tends to the boundary of $\mathcal{T}$.

As $z$ tends to a finite boundary point of $\mathcal{T}$, it is clear that $\rho_{\mathcal{T}}(z)$ tends to infinity. However, since $\overline{\mathcal{T}} \subset H$, $\rho_H(z)$ is bounded above.

Finally we need to consider the case when $z \rightarrow \infty$. In this case the standard estimate \eqref{eq.hypest} gives that $\rho_{\mathcal{T}}(z)$ is bounded below by $1/2\pi$, since no component of $\mathcal{T}$ meets a $2 \pi i$ translate of itself. The standard estimate also gives that $\rho_{H}(z)$ tends to zero. This completes the proof of the first statement.

For the second statement, suppose that $z_1, z_2$ lie in the same tract $T$. Let $\gamma \subset H$ be the hyperbolic geodesic from $F(z_1)$ to $F(z_2)$. It follows from the first statement that $F_T^{-1}(\gamma)$ has hyperbolic length at most $$d_H(F(z_1),F(z_2))/L.$$ (Recall that the notation $F_T^{-1}$ was defined at the end of Subsection~\ref{subs:properties}). The result follows.
\end{proof}
A simple but illustrative consequence of this result is the following, which is a consequence of \cite[Corollary 2.9]{lassearclike}.
\begin{lemma}
\label{lemm:bounded}
Suppose that $F \in \Blog$ is of disjoint type and $\s$ is an admissible address. Then there is at most one point in $J_{\s}(F)$ with a bounded orbit.
\end{lemma}
\begin{proof}
By way of contradiction, suppose that $z_1, z_2 \in J_{\s}(F)$ have a bounded orbit. It follows from the latter assumption that there is $K>0$ such that 
\[
d_H(F^k(z_1),F^k(z_2)) \leq K, \qfor k \in \No.
\]
Repeated application of Lemma~\ref{lemm:unifexpansion} shows that
\[
d_H(z_1, z_2) \leq L^{-k} K, \qfor k\in \No.
\]
(Observe that we are using here the principle that a map which is univalent and uniformly expanding has an inverse which is uniformly contracting.) It follows that $z_1 = z_2$. 
\end{proof}
\subsection{The Julia set of a disjoint type function}
Note that we have shown in Theorem~\ref{theo:disjointFatou} that the Julia set of $f$ is exactly those points whose orbit never lands in $D$. In particular, we have that $J(f) = \exp J(F)$, as promised. Hence, at least topologically, the study of $J(f)$ is equivalent to the study of $J(F)$. Notice that we also have that $\exp I(F) = I(f)$.

The following result is part of \cite[Proposition 2.10]{lassearclike}; see also \cite[Theorem B]{MR2290468}. We use this to show that $J(f)$ has uncountably many components, as well as illustrate the use of uniform expansion. An external address is called \emph{bounded} if it contains only finitely many symbols.
\begin{lemma}
\label{lemm:boundedorbits}
Suppose that $F \in \Blog$ is of disjoint type, and that $\underline{s}$ is a bounded external address. Then there is a unique point $z_0 \in J_{\underline{s}}(F) \setminus I(F)$. 
\end{lemma}
\begin{proof}
Fix a point $\zeta \in H \setminus \mathcal{T}$. For each of the (finitely many) tracts $s_k$ in the address $\underline{s} = s_0 s_1 \ldots$, let $\Gamma_k \subset H \setminus J(F)$ be a smooth curve connecting $\zeta$ to $F^{-1}_{s_k} (\zeta)$. Let $K>0$ be sufficiently large that $\ell_H(\Gamma_k) \leq K$, for $k\in\No$. 

Set $\gamma_0 := \Gamma_0$ and then 
\[
\gamma_k := F^{-1}_{s_0} (F^{-1}_{s_1} (\ldots (F^{-1}_{s_{k-1}}(\Gamma_k)) \ldots )), \qfor k \in \N.
\]
See Figure~\ref{f4}.

\begin{figure}
	\includegraphics[width=14cm,height=10cm]{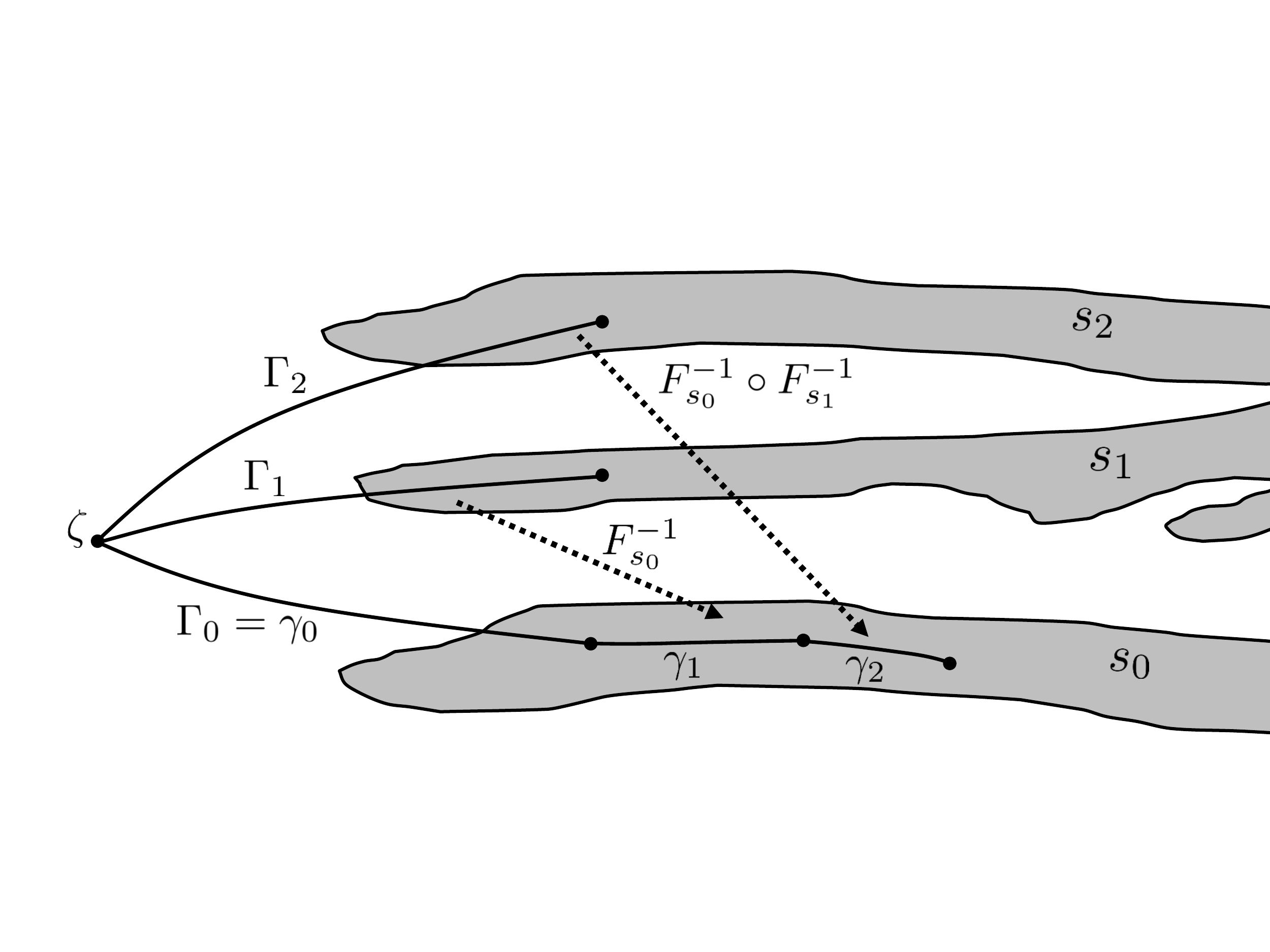}
  \caption{An illustration of the construction in the proof of Lemma~\ref{lemm:boundedorbits}.}\label{f4}
\end{figure}

It is easy to see that $\gamma_{k-1}$ and $\gamma_k$ share an endpoint. We also have, by Lemma~\ref{lemm:unifexpansion}, that $\ell_H(\gamma_{k+1}) \leq L^{-1} \ell_H(\gamma_k)$, for $k \in \No$, and so the hyperbolic lengths in $H$ of the $\gamma_k$ are bounded by a geometric sequence. It follows that 
\[
\gamma := \bigcup_{k \in\No} \gamma_k,
\]
is a piecewise smooth curve in $H \setminus J(F)$, beginning at $\zeta$, and having hyperbolic length in $H$ at most $\frac{KL}{L-1}$. We deduce that $\gamma$ has a finite endpoint $z \ne \zeta$, and moreover that $z \in J_{\underline{s}}(F)$. In addition, for each $k \in \No$, the curve $F^k(\bigcup_{j \geq k} \gamma_k)$ connects $\zeta$ to $F^k(z)$. Since this curve also is of hyperbolic length at most $\frac{KL}{L-1}$, we deduce that $z$ has a bounded orbit. 

The fact that $z$ is unique follows in a very similar way to the proof of Lemma~\ref{lemm:bounded}.
\end{proof}

We are now able to prove the following elementary topological characterisation of the Julia set of a disjoint type function, which is \cite[Corollary 2.11]{lassearclike}.
\begin{theorem}
\label{theo:Juncountable}
Suppose that $f$ is of disjoint type. Then the components of $J(f)$ are closed and unbounded. If, in addition, $f$ is of disjoint type, then $J(f)$ has uncountably many components.
\end{theorem}
\begin{proof}
Connected components of a closed set are closed. It is well-known \cite[Proposition 2]{MR1196102} that a function $f \in\B$ has no multiply-connected Fatou components, and so $J(f) \cup \{\infty\}$ is connected. It follows by a standard result of continuum theorem, known as the ``boundary bumping theorem'' (see \cite[Theorem 5.6]{MR1192552}), that every connected component of $J(f)$ is unbounded. 

Now suppose that $f$ is disjoint type, let $F$ be a logarithmic transform of $f$, and let $T$ be a tract of $F$. There are uncountably many bounded external addresses of $F$ with initial entry $T$. By Lemma~\ref{lemm:boundedorbits} these addresses each correspond to a different component of $J(F)$. Each component of $J(F)$ corresponds to a component of $J(f)$ under the exponential map, and since $\exp$ is injective on $T$, these are pairwise disjoint.
\end{proof}
Observe that if $f \in \B$, and $J$ is a component of $J(f)$, then $\hat{J} = J \cup \{\infty\}$ is a continuum; we term these sets the \emph{Julia continua} of $f$. Note that these sets are not to be confused with the Julia continua of $F$ defined in \eqref{eq:jcdef}.
\subsection{The topology of Julia set continua of a disjoint type function}
Rempe-Gillen \cite[Theorem 1.6]{lassearclike} has completely characterised the topology of Julia continua of a disjoint type function. Extremely roughly speaking, one might hope that, by pulling back compact sets as in \eqref{eq:pullback}, we might obtain something resembling an arc. In fact this is true, in a way that can be made precise. However, in order to do this we need some technical definitions from continuum theory; see, for example, \cite{MR1192552} for further background on this topic.
\begin{definition}
Suppose that $C$ is a continuum (in other words, a non-empty compact connected set). Then:
\begin{itemize}
\item $C$ \emph{has span zero} if any subcontinuum $A \subset C \times C$ whose first and second coordinates both project to the same subcontinuum of $C$, also contains a point of the form $(\xi, \xi)$. 
\item $x \in C$ is a \emph{terminal point} if, for any two subcontinua $A,B \subset C$ with $x \in A \cap B$, either $A \subset B$ or $B \subset A$. 
\item $C$ is \emph{arc-like} if for each $\epsilon > 0$ there is a continuous function $\phi : C \to [0, 1]$ such that the Euclidean diameter of the set $\phi^{-1}(t)$ is less than $\epsilon$, for $t \in [0,1]$.
\end{itemize} 
\end{definition}
As these definitions are quite complicated, we give the following rough, but more intuitive interpretations:
\begin{itemize}
\item A continuum $C$ has span zero if, when we try to exchange the positions of any two points in $C$ by moving them through $C$, we cannot do so without the points coinciding at some stage. For example, an arc has span zero, a circle does not.
\item Terminal points are a natural analogue of the endpoints of an arc. However, unlike an arc, a continuum may have many more than two terminal points.
\item A continuum is arc-like if it looks like a ``blurred'' arc at all levels of magnification. See Figure~\ref{farc}.
\end{itemize} 

\begin{figure}
	\includegraphics[width=14cm,height=10cm]{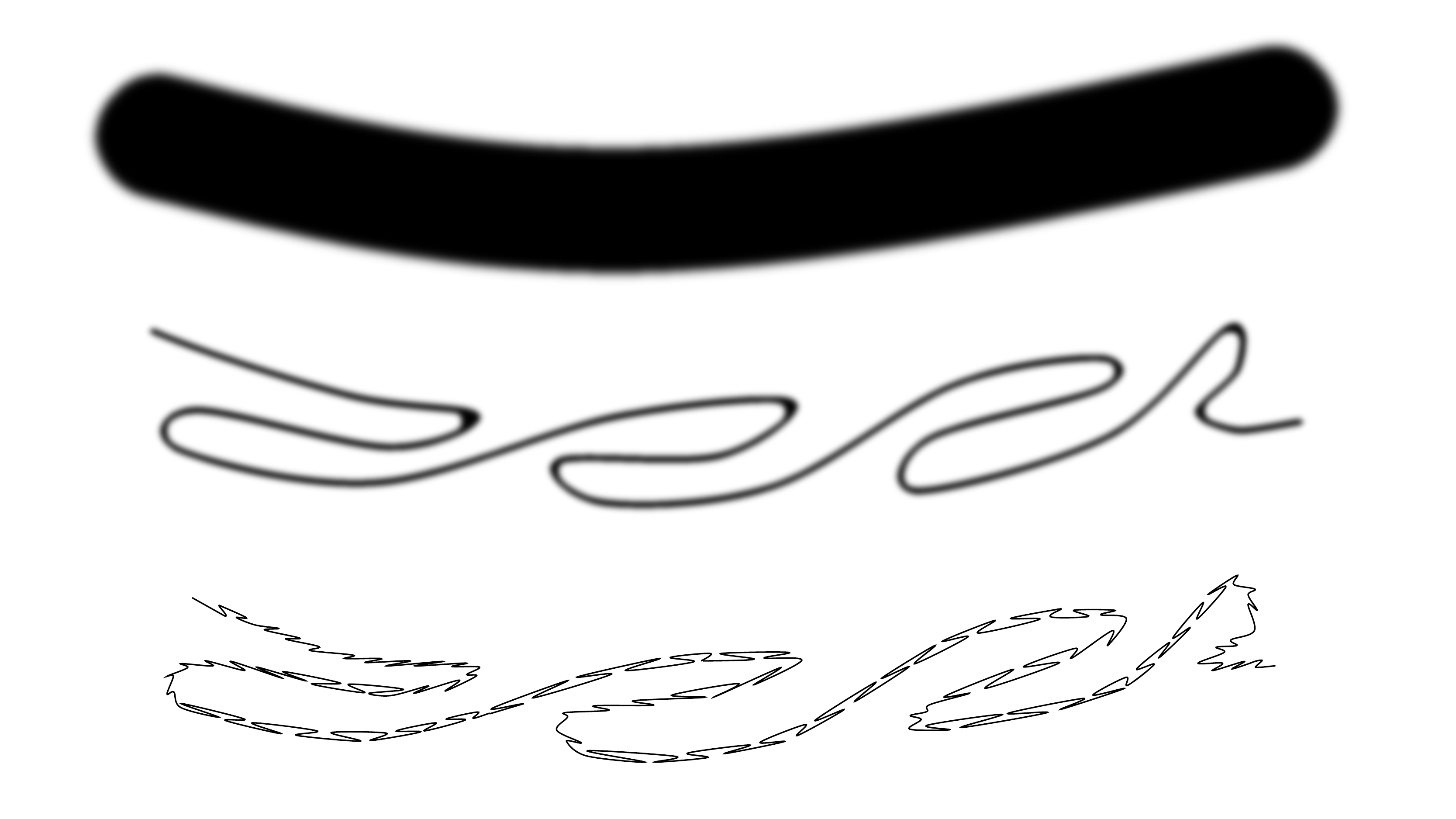}
  \caption{An illustration of an arc-like continuum, shown at three different ``magnifications''.}\label{farc}
\end{figure}

We also need to introduce a purely geometric constraint on tracts. If $F \in \Blog$, then we say that $F$ \emph{has bounded slope} if there exists a curve $\gamma : [0,\infty) \to \mathcal{T}$, which tends to $+\infty$, and a $C>0$, such that $$|\operatorname{Im}z| \leq C \operatorname{Re} z, \qfor z \in \gamma.$$ 
See Figure~\ref{fbad}. If $f \in \B$ has a logarithmic transform $F$ with bounded slope, then we say that $f$ has bounded slope.

We can now give a precise statement of Rempe-Gillen's result.
\begin{theorem}
\label{theo:lassearclike}
Suppose that $f \in \B$ is of disjoint type, and that $\hat{J}$ is a Julia continuum of $f$. Then $\hat{J}$ has span zero, and $\infty$ is a terminal point of $\hat{J}$. If,
in addition, f has bounded slope, then $\hat{J}$ is arc-like.
\end{theorem}
\begin{remark}\normalfont
It can be shown to follow from the fact that $\infty$ is a terminal point of $\hat{J}$, that each set $J_{\s}(F)$ is connected, as mentioned earlier. It then follows that if $f \in \B$ is of disjoint type, and $F$ is a logarithmic transform of $f$, then every Julia continuum of $F$ is homeomorphic (via the exponential function) to a Julia continuum of $f$.
\end{remark}
\begin{proof}[Proof of Theorem~\ref{theo:lassearclike}]
We refer to \cite{lassearclike} for detailed proofs of these results, which all rely on the fact that any logarithmic transform, $F$, of a disjoint type map has uniform expansion in the hyperbolic metric. 

So that we can work with the transform, let $J$ be an unbounded continuum such that $\exp (J\setminus\{\infty\}) = \hat{J}\setminus\{\infty\}$. Since topological properties are preserved by the exponential map, it is sufficient to prove that these properties hold for $J$.
 
We first sketch a proof that $\infty$ is a terminal point of $J$. Suppose that $A, B$ are subcontinua of $J$ that contain $\infty$. Without loss of generality we can assume that, for infinitely many $n \in \N$, $F^n(A)$ contains a point of real part less than or equal to all the real parts of points of $F^n(B)$. We claim that we can deduce from this that $B \subset A$; it is clear that the result follows.

Suppose that $z \in B$. By a geometric argument, it can be shown that for each value $n$ above, there is a point $z_n \in A$ such that $|F^n(z_n) - F^n(z)| \leq 2\pi$; essentially this follows from the fact that the $F^n(A)$ and $F^n(B)$ lie in the same tract. Pulling back, and using Lemma~\ref{lemm:unifexpansion}, it follows that there exists a sequence of points of $A$ that tend to $z$. Since $A$ is closed, it follows that $z \in A$, as required.

Next we sketch, very roughly, a proof that $J$ has span zero. No tract $T$ can intersect a $2 \pi i$ translate of itself. It follows that two points cannot exchange position by moving inside $T$ without coming within a distance of $2\pi$ from each other. Let $\underline{s} = s_0 s_1 \ldots$ be the external address of $J$. By applying this observation to the tract $s_n$, for $n$ large, and using the expanding property of $F$, we see that two points cannot cross each other within $J$ without passing within distance $\epsilon$ of each other, for all $\epsilon > 0$. The result follows.

We omit the proof of the final statement.
\end{proof}
In fact, in Theorem~\ref{theo:lassearclike} we have only quoted half of \cite[Theorem 1.6]{lassearclike}. The other half of \cite[Theorem 1.6]{lassearclike} shows that, essentially, Theorem~\ref{theo:lassearclike} is strong. This result is as follows.
\begin{theorem}
\label{theo:lassearclikeexample}
There is a disjoint type function $f \in \B$, of bounded slope, with the following property. If $X$ is any arc-like continuum with a terminal point $x$, then there exists a Julia continuum, $\hat{J}$, of $f$, and a homeomorphism $\psi : X \to \hat{J}$ such that $\psi(x) = \infty$.
\end{theorem}
Roughly speaking, this remarkable result says that the single function $f$ has a Julia continuum homeomorphic to every Julia continuum permitted by Theorem~\ref{theo:lassearclike}; and there are uncountably many of these.
\begin{proof}[The idea behind the proof of Theorem~\ref{theo:lassearclikeexample}]
We do not attempt to prove this theorem, which requires a number of results from continuum theory, and is somewhat technical. The crux of the proof is the observation that we can create Julia continua of a function in the class $\Blog$, with certain properties, by ``drawing'' sufficiently complicated tracts and then applying a Riemann map. This idea can be made precise.

Rempe-Gillen classifies each arc-like continuum with a terminal point in terms of a so-called \emph{inverse limit}, and then shows how to ``draw'' the correct tract based on the inverse limit. Julia continua of a function in the class $\B$ are then obtained by using the construction of Bishop; see Section~\ref{s.bishop}.
\end{proof}
%
%
%
%
%
\section{Finite order functions in the class $\B$}
\label{s.finorder}
In this section we discuss a different subclass of the class $\B$; those functions of finite order. At face value this seems an unexpected property to study in a dynamical setting since, in general, it is not preserved under iteration. However will see that some particularly strong and important results hold for these functions that do not hold in general. In particular, we will show that for functions $f\in\B$ of finite order, all components of $I(f)$ are path-connected and unbounded. We will then show that if $f$ is also of disjoint type, then its Julia set has a distinctive topological structure known as a Cantor bouquet. Finally we will see that for finite order functions in the class $\B$, the result of Theorem~\ref{theo:dimension} can be improved significantly.
\subsection{Eremenko's conjecture}
In \cite{MR1102727} Eremenko studied the escaping set $I(f)$ of a {\tef} $f$. He showed that all components of $\overline{I(f)}$ are unbounded, and conjectured that all components of $I(f)$ are unbounded; this conjecture, known as \emph{Eremenko's conjecture} is still open, despite much progress. He also conjectured that every point of $I(f)$ can be joined to infinity by a curve in $I(f)$; this conjecture is known as the \emph{strong version of Eremenko's conjecture}.

In \cite[Theorem 1.1]{MR2753600} the authors construct a disjoint type function $f \in \B$ such that all path-connected components of $J(f)$ -- and so all path-connected components of $I(f)$ -- are bounded, showing that the strong version of Eremenko's conjecture does not hold in general. In fact, it is even possible to construct this function so that $I(f)$ contains no arcs. Very roughly, the authors define a tract which ``wiggles'' in a very complicated way, and use a Riemann map to obtain a function $F \in \Blog$ of disjoint type. It is then shown that the geometry of the tract leads to all path-components of $I(F)$ being singletons. A function $f \in \B$ with the required properties is then constructed from $F$. (This approach was, in fact, a precursor to the more complicated construction in Theorem~\ref{theo:lassearclikeexample}.)

In the same paper the authors showed that the strong version of Eremenko's conjecture does hold for a certain subclass of the class $\B$, which we introduce next.
\subsection{Finite order functions}
First we give a definition. We say that a {\tef} $f$ has \emph{finite order} if
\[
\log \log |f(z)| = O(\log |z|), \quad\text{as } |z| \rightarrow\infty.
\]
This property can easily be translated to the class $\Blog$. We say that $F\in\Blog$ has \emph{finite order} if
\[
\log \operatorname{Re} F(z) = O(\operatorname{Re} z), \quad\text{as } \operatorname {Re} z \rightarrow\infty \text{ in } \mathcal{T}.
\]
It is easy to see that if $F$ is a logarithmic transform of $f$, then $f$ has finite order exactly when $F$ has finite order.

Many of the well-known examples of class $\B$ functions are of finite order; all maps in the exponential or cosine families, for example, are of finite order. The next result \cite[Theorem 1.2]{MR2753600} shows that the strong version of Eremenko's conjecture holds for finite order functions in the class $\B$. Note that, for functions which are also of disjoint type, this was also proved in \cite{MR2318569}. 
\begin{theorem}
\label{theo:finiteorder}
Suppose that $f \in \B$ is of finite order. Then every point $z \in I(f)$ can be connected to infinity by a curve $\gamma \subset I(f)$, on which the iterates of $f$ tend to infinity uniformly.
\end{theorem}
\begin{proof}
The proof of this result proceeds in four stages. We outline the ideas without providing all the details necessary for a rigorous proof.

First, let $F$ be a logarithmic transform of $f$. It is shown that the tracts of $F$ have two geometric properties which constrain the dynamics from being too ``pathological''; clearly the tracts constructed in the proof of \cite[Theorem 1.1]{MR2753600} do not have these geometric properties. The first property is that the tracts have bounded slope, which we defined earlier. The second is that the tracts have \emph{uniformly bounded wiggling}. Geometrically speaking, a tract has bounded wiggling if it cannot ``double back'' on itself by too much.
\begin{definition}
Let $F \in \mathcal{B}_{\log}$. A tract $T$ of $F$ has \emph{bounded wiggling} with constants $K> 1$ and $\mu >0$ if, for each point $w_0\in \overline{T}$, every point $w$ on the hyperbolic geodesic of $T$ that connects $w_0$ to $\infty$ satisfies
$$(\operatorname{Re}w)^+ > \frac{1}{K} \operatorname{Re}w_0 - \mu.$$
(Here $t^+ := \max\{t,0\}$, for $t \in \R$.) 
If all the tracts of $F$ have bounded wiggling for the same constants, then we say that the tracts have \emph{uniformly bounded wiggling}.
\end{definition}

The function $F$ can be shown to have tracts with bounded slope and uniformly bounded wiggling by an argument from the fact it is of finite order, together an application of hyperbolic geometry; we omit the detail.

\begin{figure}
	\includegraphics[width=14cm,height=8cm]{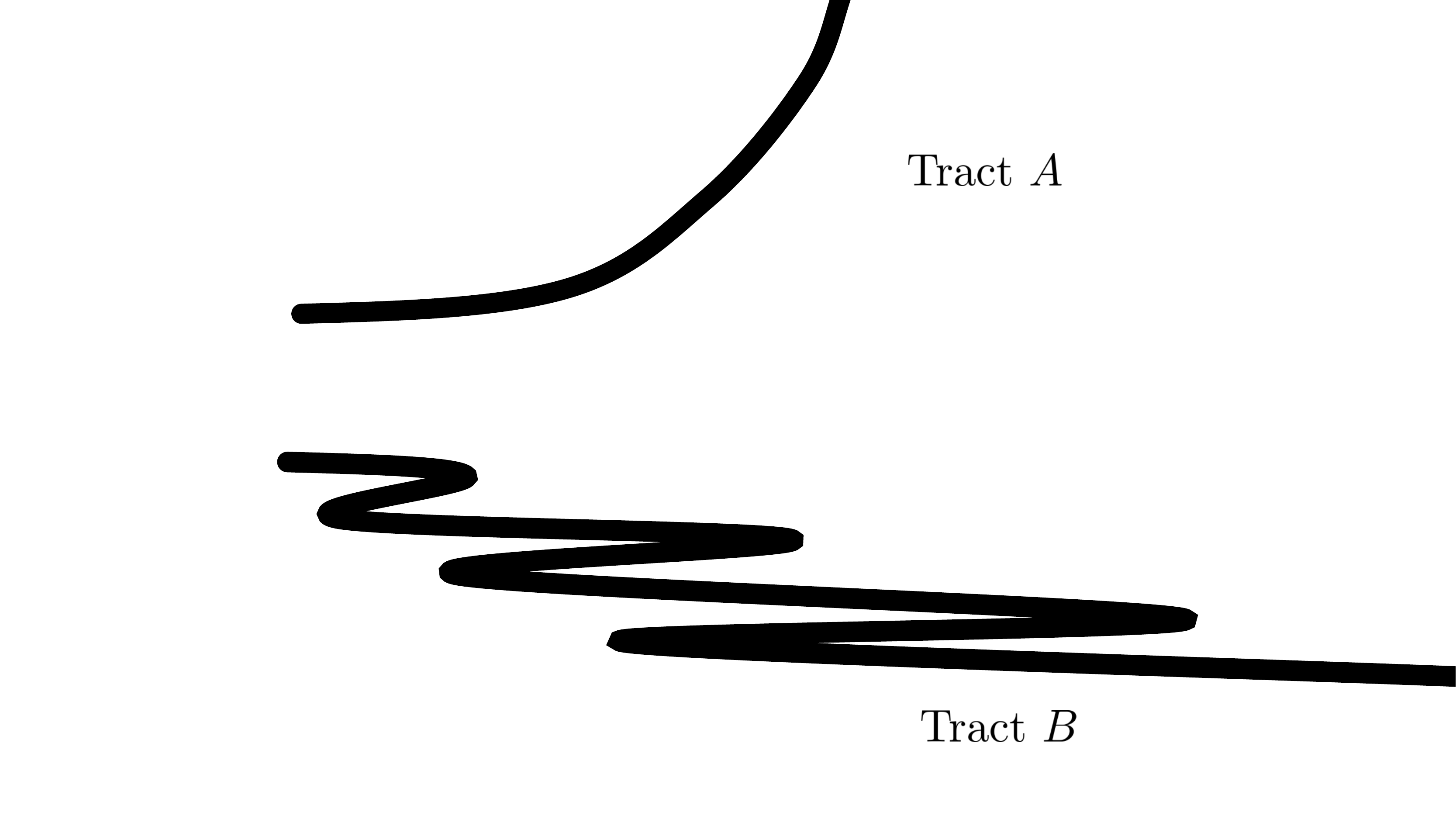}
  \caption{Tract $A$ does not have bounded slope. Tract $B$ does not have bounded wiggling. These configurations are both impossible for a function of finite order.}\label{fbad}
\end{figure}

\begin{remark}\normalfont
We note here that, in fact, the rest of the proof can be completed for any $f \in \B$ that has tracts with bounded slope and uniformly bounded wiggling. This is the case, for example, if $f$ is a finite composition of functions in class $\B$ each of finite order.
\end{remark}

The second stage is to show that these two geometric conditions imply that $F$ satisfies a so-called \emph{uniform linear head-start condition}. This condition holds if there exist $K>1$ and $M>0$ with the following property. If $z_1, z_2$ are points in the closure of some tract ${T}$, such that $F(z_1), F(z_2)$ are in the closure of some tract ${T'}$, then
\begin{equation}
\label{eq:roughly}
\operatorname{Re} z_2 > K(\operatorname{Re} z_1)^+ + M \implies \operatorname{Re} F(z_2) > K(\operatorname{Re} F(z_1))^+ + M.
\end{equation}
Once again the argument is from hyperbolic geometry. If the tract is a horizontal strip, then the argument at the end of Subsection~\ref{subs:hyperbolic} can be used to complete the result; very roughly speaking, the condition on the left-hand side of \eqref{eq:roughly} says that $z_1$ and $z_2$ are a large hyperbolic distance apart in $T$, and then the right-hand side of \eqref{eq:roughly} can then be deduced from the fact that their images must be a large distance apart in $H$. The argument in generality is similar, although more complicated.

The third stage is to show that a head-start condition induces an ordering on each Julia continuum of $F$. (Note that, for simplicity, we are now assuming that $F$ is disjoint type, and so the Julia continua exist as defined earlier. The case that $F$ is not of disjoint type is only slightly more complicated.) 

Let $\underline{s}$ be any admissible external address. For each $z_1, z_2 \in J_{\underline{s}}(F)$ we write $z_1 \succ z_2$ if there exists $k \in \N$ such that
\[
\operatorname{Re} F^k(z_1) > K(\operatorname{Re} F^k(z_2))^+ + M,
\]
where $K, M$ are the constants from the uniform linear head-start condition. Setting $\infty \succ z$, for $z \in J_{\underline{s}}(F)$, it can be shown that $(\hat{J}_{\underline{s}}(F), \succ)$ is a totally ordered space, and that the order topology on $(\hat{J}_{\underline{s}}(F), \succ)$ agrees with the metric topology on $\hat{J}_{\underline{s}}(F)$. (An ordered space $(X, \succ)$ is \emph{totally ordered} if for any two $x, y \in X$ either $x \succ y$ or $y \succ x$.) It follows by a result in continuum theory that $\hat{J}_{\underline{s}}(F)$ is homeomorphic to a compact interval. Moreover, it can be shown that if $z_1 \succ z_2$ and $z_1 \ne \infty$, then $z_1 \in I(F)$. The result for $F$ follows at once.

To complete the proof, suppose that $z \in I(f)$. (Recall the sets $D$ and $W$ in Figure~\ref{f1}). The orbit of $z$ must eventually leave $D$, so choose $z' = f^k(z)$ so that $f^n(z') \in W$, for $n \in \N$. Choose a point $w$ such that $z' = \exp w$; then $w \in I(F) \subset J(F)$. We can let $\underline{s}$ denote the address of $w$. We then apply the above argument to $J_{\underline{s}}$ to obtain a curve $\gamma' \subset I(F)$ joining $w$ to $\infty$. Finally let $\gamma$ be the component of $f^{-k}(\exp(\gamma'))$ containing $z$.
\end{proof}
This result was strengthened in \cite[Theorem 1.2]{MR2675603}, where it was shown that all the points of $\gamma$, except possibly $z$ itself, lie in the \emph{fast escaping set}, first defined in \cite{MR1684251}. This is a much-studied subset of the escaping set, containing all those points that escape to infinity as fast as possible (in a sense that can be made precise). We omit further detail, but refer the interested reader to \cite{Rippon01102012} which gives a detailed study of this set.
\subsection{Cantor bouquets}
Suppose that $f \in \B$ has finite order. Theorem~\ref{theo:finiteorder} says that the escaping set -- and hence the Julia set -- contains curves. However, it tells us very little about the topology of these curves.

If we make the additional assumption that $f$ is of disjoint type, then we can make a very clear statement about the topology of the Julia set. To this end, we need to define a topological structure known as a \emph{Cantor bouquet}. In fact a Cantor bouquet is defined as a subset of $\C$ that is ambiently homeomorphic to a straight brush, so it is this latter object that we need to define; see \cite{MR1182980} and also Figure~\ref{fbrush}. (Note that two sets $A, B \subset \C$ are ambiently homeomorphic if there is a homeomorphism $\phi : \C \to \C$ such that $\phi(A) = B$.)

\begin{figure}
	\includegraphics[width=13cm,height=7cm]{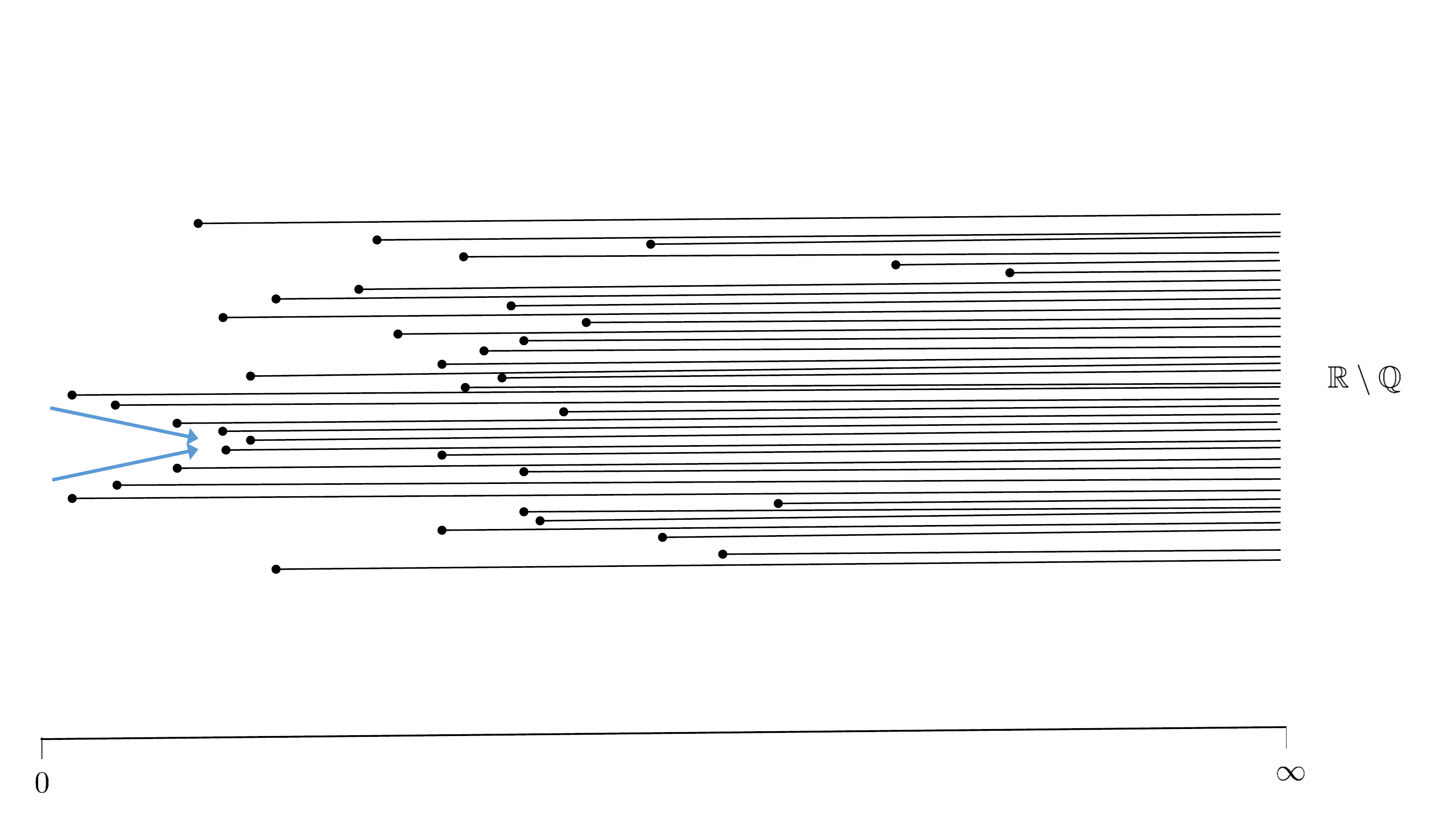}
  \caption{An illustration of (part of) a straight brush. The arrows indicate some ``hairs'' tending to another hair ``from below'' and ``from above''.}\label{fbrush}
\end{figure}

\begin{definition}
Suppose that $B$ is a closed subset of $[0, +\infty) \times (\R \setminus \Q)$. For convenience, let $H := \{ y : (x, y) \in B \text{ for some } x \}$. Then $B$ is called a \emph{straight brush} if
\begin{enumerate}[(a)]
\item The set $H$ is dense in $(\R \setminus \Q)$.\label{brush1}
\item For each $y \in H$, there exists $t_y \geq 0$ such that $\{ x : (x, y) \in B \} = [t_y, +\infty).$ We call $[t_y, +\infty) \times y$ the \emph{hair attached} at $y$. We call $(t_y, y)$ the \emph{endpoint}.\label{brush2}
\item For each $(x, y) \in B$ there are two sequences of hairs attached at $a_n, b_n$ respectively, such that $a_n < y < b_n$, and, as $n \rightarrow\infty$, $a_n, b_n \rightarrow y$, and $t_{a_n}, t_{b_n} \rightarrow t_y$.\label{brush3}
\end{enumerate}
\end{definition}

%
We then have the following exact topological description of the Julia set of a disjoint type function of finite order; see \cite[Theorem 1.5]{MR2902745}.
\begin{theorem}
\label{theo:cantorbouquet}
If $f$ is a {\tef} of disjoint type and finite order, then $J(f)$ is a Cantor bouquet.
\end{theorem}
\begin{proof}
Roughly speaking, the proof of Theorem~\ref{theo:finiteorder} gives rise to the hairs in part (b) of the definition of a straight brush. It remains to show parts \eqref{brush1} and \eqref{brush3}. Part \eqref{brush1} is shown by proving that the set of admissible external addresses is dense in the set of external addresses. In fact, this is almost an immediate corollary of Lemma~\ref{lemm:boundedorbits}, since the set of periodic external addresses is dense in the set of external addresses, and any periodic external address is clearly bounded and hence admissible.

To prove \eqref{brush3}, suppose that $z \in J(F)$ has external address $\underline{s}$. For each $n \in \N$, we let $\phi_n$ denote the branch of $F^{-n}$ that maps $F^n(z)$ to $z$, and we set $z_n^{\pm} = \phi_n(F^n(z) \pm 2 \pi i)$. In other words, $z^+_n$ (resp. $z^-_n$) has the same first $n$ symbols in its external address as $z$, but the next symbol is shifted ``up'' (resp. ``down'') $2 \pi i$. It can then be shown, using the uniform contraction of the inverse we have used before, that the hairs containing $z^+_n$ (resp. $z^-_n$) accumulate on the hair containing $z$ from above (resp. below), as $n \rightarrow\infty$, and then that this is sufficient to prove \eqref{brush3}.
\end{proof}

We give an illustration of the power of Theorem~\ref{theo:cantorbouquet}. For functions $$f(z) = \lambda e^z, \qfor \lambda \in (0, 1/e),$$ referred to earlier, Mayer \cite{MR1053806} studied the set $E(f)$ of endpoints of the rays in the Julia set. He showed the memorable result that $E(f) \cup \{\infty\}$ is connected, but $E(f)$ itself is totally separated; a set $X \subset \C$ is totally separated if for points $a, b \in X$ there exists a relatively open and closed set $U \subset X$ such that $a \in U$ and $b \notin U$. In other words, $\infty$ is an \emph{explosion point} for the set $E(f) \cup \{\infty\}$.

Note that the functions studied by Mayer are all of disjoint type and finite order. With the definition of a Cantor bouquet given above, the union of a Cantor bouquet with infinity can be shown to be a topological object called a \emph{Lelek fan}; this object was first defined in \cite{MR0133806} and this result is \cite[Theorem 2.8]{MR3620879}. It is a topological result (see, for example, \cite{MR1002079}) that infinity is an explosion point for the set of endpoints of a Lelek fan together with infinity. In other words, the following generalisation of Mayer's result is an immediate, purely topological, corollary of Theorem~\ref{theo:cantorbouquet}.
\begin{corollary}
If a {\tef} $f$ is of disjoint type and finite order, then infinity is an explosion point of $E(f) \cup \{\infty\}$.
\end{corollary}
We note all Julia set points which are not escaping must be endpoints \cite[Theorem 4.7]{MR2753600}. For more discussion on endpoints, and their topological properties, see \cite{MR3620879, VassoDave}. \\

Suppose that $f \in \B$ is of finite order and disjoint type, and let $F$ be a logarithmic transform of $f$. We know that $J(f) = \exp J(F)$ and, in a sense, the Cantor bouquet in Theorem~\ref{theo:cantorbouquet} arises as a consequence of this. It is natural to ask the following. Suppose that we now relax the condition that $f$ is of disjoint type. By considering the set $\exp J(F)$ does a Cantor bouquet arise as a subset of $J(f)$? In fact this is true, as shown in the following result \cite[Theorem 1.6]{MR2902745}, the proof of which is omitted.
\begin{theorem}
If a {\tef} $f \in \B$ is of finite order, then there is a Cantor bouquet $X \subset J(f)$ such that $f(X) \subset X$.
\end{theorem}
\subsection{The dimension of the Julia set of a finite order function}
We recall Theorem~\ref{theo:dimension}, which states that the Julia set of a class $\B$ function has Hausdorff dimension greater than $1$. In the case of a function in this class of finite order, a much stronger result holds.
\begin{theorem}
\label{theo:equalstwo}
If $f \in \B$ is of finite order, then $\dim_H J(f) = 2$.
\end{theorem}
This result is due to Bara\'{n}ski \cite[Theorem A]{MR2464786} and to Schubert \cite[Satz 1.1.1]{schubertthesis}. The technique of the proof is somewhat similar to that of Theorem~\ref{theo:dimension}, and we omit the details. This result was later generalised in \cite{MR2559123}. It is worth noting that the hypothesis that $f$ is of finite order cannot be omitted from Theorem~\ref{theo:equalstwo}. Stallard \cite{MR1760674}, using an explicit construction, showed that for each $d \in (1,2]$ there is a function $f_d \in \B$ such that $\dim_H J(f_d) = d$.
\section{Hyperbolic functions}
\label{s.hyperbolic}
In this section we discuss a class of maps that is much more general than the class of disjoint type maps, but with properties that ensure their dynamical properties are amenable to study. 
\subsection{Definition of hyperbolic functions}
Recall that the definition of disjoint type maps was very intuitive. The definition of hyperbolic functions seems at first to be somewhat less so. To support our definition, we will motivate a definition of hyperbolic {\tef}s using ideas which are familiar in other areas of dynamics. We will then see that, in fact, the only {\tef}s that are hyperbolic are in class $\B$.

It is a general principle in the investigation of dynamical systems that \emph{hyperbolic systems} (sometimes, following Smale \cite{smaledynamics}, known as ``Axiom A'') are the first class to understand: they show the simplest behaviour, yet their study leads to a better understanding in greater generality. For rational maps, including polynomials, there is a well-established definition of hyperbolicity. A rational map $f : \widehat{\C} \to \widehat{\C}$ is said to be \emph{hyperbolic} if one of the following equivalent conditions holds (\cite[Section~9.7]{beardon}, see also \cite[Chapter 19]{MR2193309}):
\begin{enumerate}[(a)]
\item The function $f$ is \emph{expanding} with respect to a suitable conformal metric defined on a neighbourhood of its Julia set. 
\item Every critical value of $f$ belongs to the basin of an attracting periodic cycle.
\item The postsingular set is a subset of the Fatou set.
\end{enumerate}

We highlight the definition of ``expanding'' above; expansion on a neighbourhood of the Julia set is a key factor in determining the dynamics of a hyperbolic map. For a {\tef} it is more difficult to get the definition of ``expanding'' right. The following was proposed in \cite{MR3671560}; although this is more complicated than (a) above, stronger versions of this definition exclude functions we would want to call hyperbolic, and weaker definitions include functions we would not.
\begin{definition}
A {\tef} $f$ is \emph{expanding} if there exist a connected open set $W\subset\C$, which contains $J(f)$, and a conformal metric $\rho=\rho(z)|dz|$ on $W$ such that:
\begin{enumerate}[(1)]
\item $W$ contains a punctured neighbourhood of infinity;
\item $f$ is expanding with respect to the metric $\rho$, i.e. there exists
       $\lambda>1$ such that 
          \[ \|Df(z)\|_{\rho}  \geq \lambda, \qfor z, f(z) \in W; \]
\item the metric $\rho$ is complete at infinity, i.e. $\operatorname{dist}_\rho(z, \infty) = \infty$ whenever $z \in W$.
\end{enumerate}
\end{definition}
We then have the following result \cite[Theorem 1.3]{MR3671560}, which both provides a definition of hyperbolic functions and explains why such functions are only found in the class $\mathcal{B}$. 
\begin{theorem}
Suppose that $f$ is a {\tef}. Then the following are equivalent.
\begin{enumerate}[(a)]
\item $f$ is expanding.\label{hyp3}
\item $f \in \B$, and every point in $S(f)$ lies in the basin of an attracting periodic cycle of $f$.\label{hyp1}
\item $P(f)$ is a compact subset of $F(f)$.\label{hyp2}
\end{enumerate}
A function which satisfies any, and hence all, of these properties is said to be \emph{hyperbolic}.
\end{theorem}
\begin{proof}
We very briefly sketch a part of the proof of this result. Suppose first that \eqref{hyp1} holds. Exactly as in the proof of Theorem~\ref{theo:disjointFatou}, we can show that the only Fatou components of $f$ are basins of attraction. Part  \eqref{hyp2} can then be deduced from properties of attracting basins.

Next, suppose that \eqref{hyp2} holds. We can create a bounded open neighbourhood, $U$, of $P(f)$, such that $\overline{f(U)} \subset U$. Set $W := \C \setminus \overline{U}$. Property \eqref{hyp3} can then be deduced by considering the properties of the covering map $f : f^{-1}(W) \to W$; see \cite[Lemma 5.1]{MR2570071}.

Finally suppose that \eqref{hyp3} holds. The expanding property on $W$ can be used to show that $W \cap S(f) = \emptyset$, in which case $f \in \B$. It can then be shown that  this fact, together with the expanding property of $f$ on $W$, implies that $J(f) \cap P(f) = \emptyset$. Property \eqref{hyp1} can then be deduced.
\end{proof}

It is easy to see that all disjoint type functions are hyperbolic. An example of a hyperbolic function which is not of disjoint type is the function $f(z) = \frac{\pi}{2} \sin z$. This function has $S(f) = \{ \pm \pi/2\}$. These are both attracting fixed points, and so $f$ is hyperbolic by property \eqref{hyp1}. However, we know from Theorem~\ref{theo:disjointFatou} that this behaviour is impossible for a function of disjoint type.
\subsection{The Fatou set of a hyperbolic function}
As we have seen above, the only Fatou components of a hyperbolic map are attracting basins. However, unlike disjoint type functions, hyperbolic functions may have more than one immediate attracting basin. An example from  \cite{MR3433280} is the map $$f(z) = -z^2 \exp(1-z^2),$$ for which $CV(f) = \{0, -1\}$ and $AV(f) = \{0\}$. The points $0$ and $-1$ are in fact attracting, and can be shown to lie in bounded Fatou components. On the other hand, all preperiodic components of the immediate basin of $0$ are unbounded.

It is natural to ask, then, when Fatou components of a hyperbolic function must be bounded. The following \cite[Theorem 1.2]{MR3433280} gives a complete answer to this question.
\begin{theorem}
Suppose that $f$ is a hyperbolic transcendental entire function. Then every component of $F(f)$ is bounded if and only if $f$ has no finite asymptotic values and every component of $F(f)$ contains at most finitely many critical points. 
\end{theorem}
The proof of this result is omitted.
\subsection{The Julia set of a hyperbolic function}
As noted earlier, if $f \in \B$, then the function $g(z) = \lambda f(z)$ is of disjoint type whenever $|\lambda|$ is chosen small enough. In the case that $f$ is hyperbolic, then a result of Rempe-Gillen \cite[Theorem 5.2]{MR2570071} can be used to transfer dynamical properties of $g$ back to $f$.
\begin{theorem}
\label{theo:conjugacy}
Suppose that $f, g$ are as above, and that $|\lambda|$ is sufficiently small. Then there is a continuous surjection 
\[
\vartheta{ϑ} :J(g) \to J(f),
\]
with 
\begin{equation}
\label{commutes}
f \circ \vartheta = \vartheta \circ g.
\end{equation}
Furthermore, $\vartheta{ϑ}: I(g) \to I(f)$ is a homeomorphism.
\end{theorem}

\begin{figure}
	\includegraphics[width=14cm,height=8cm]{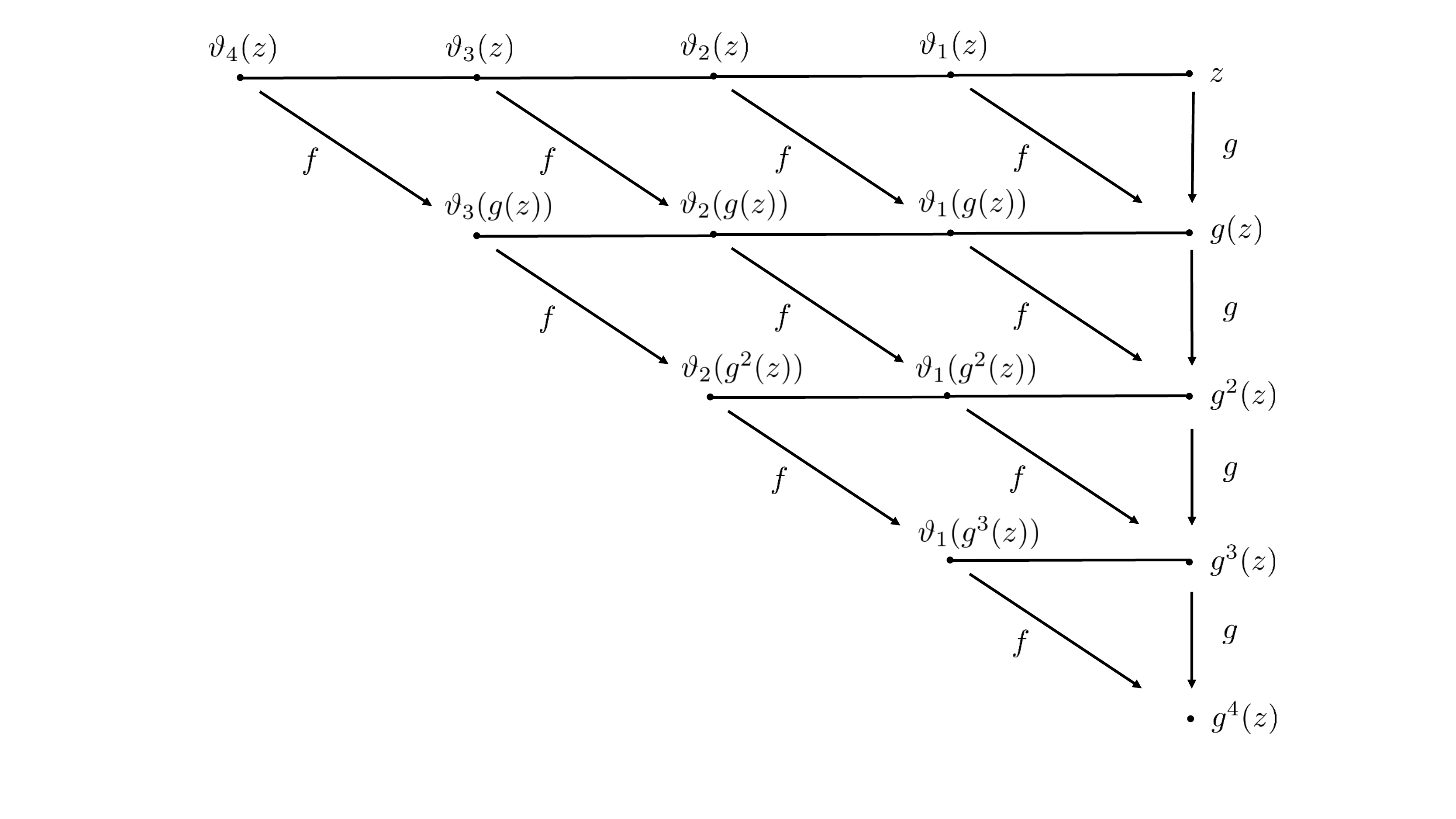}
  \caption{An illustration of the functions in the proof of Theorem~\ref{theo:conjugacy}.}\label{f3}
\end{figure}

\begin{proof}
Let $W \supset J(f)$ be the domain in the definition of a hyperbolic map. We will roughly sketch the approach to constructing the map $\vartheta$; see Figure~\ref{f3}. It is slightly easier to let $g(z) = f(\lambda z)$; since this is conjugate to the form stated above there is no loss of generality in doing this. 

We inductively define a sequence of functions $\vartheta_n(z)$, for $n \geq 0$, such that
\[
f \circ \vartheta_{n+1} = \vartheta_n \circ g, \qfor n \in \N.
\]
We begin by setting $\vartheta_0(z) := z$ and $\vartheta_1(z) := \lambda z$, for $z \in J(g)$. We show how to construct $\vartheta_2(z)$. First we join $\vartheta_0(g(z))$ to $\vartheta_1(g(z))$ with a line segment, $\gamma_1$ say. We then pull back $\gamma_1$ using the inverse branch of $f$ that maps $\vartheta_0(g(z))$ to $\vartheta_1(z)$; by making $|\lambda|$ small at the start of the construction, we can ensure that there is always a neighbourhood of $\gamma$ which does not meet $S(g)$, and so this inverse branch is well-defined and gives a new line segment $\gamma'$. One end of $\gamma'$ is at $\vartheta_1(z)$, and the other end then defines $\vartheta_2(z)$.   

This process can be continued iteratively; for example $\vartheta_3(z)$ is defined by pulling back the line segment from $\vartheta_0(g^2(z))$ to $\vartheta_1(g^2(z))$ using the correct branch of $f^{-2}$. Moreover, because $f$ is expanding on $W$ (and so the inverse is contracting), it can be shown that the maps $\vartheta_n$ in fact converge to a continuous map $\vartheta$ with the property that there exists $K>0$ such that
\begin{equation}
\label{nottoofar}
\operatorname{dist}_W(z, \vartheta(z)) \leq K, \qfor z \in J(g).
\end{equation}

Equation \eqref{commutes} is immediate. Moreover, if $z \in I(g)$, then it follows from \eqref{nottoofar} that 
\[
f^n(\vartheta(z)) = \vartheta(g^n(z)) \rightarrow \infty \text{ as } n \rightarrow \infty.
\]
Hence $\vartheta(z) \in I(f)$, and so $\vartheta(I(g)) = I(f)$. Since the Julia set is the boundary of the escaping set, it also follows that $\vartheta(J(g)) = J(f)$. The other stated properties of $\vartheta$ can then be deduced quickly.
\end{proof}

\begin{figure}
	\includegraphics[width=14cm,height=8cm]{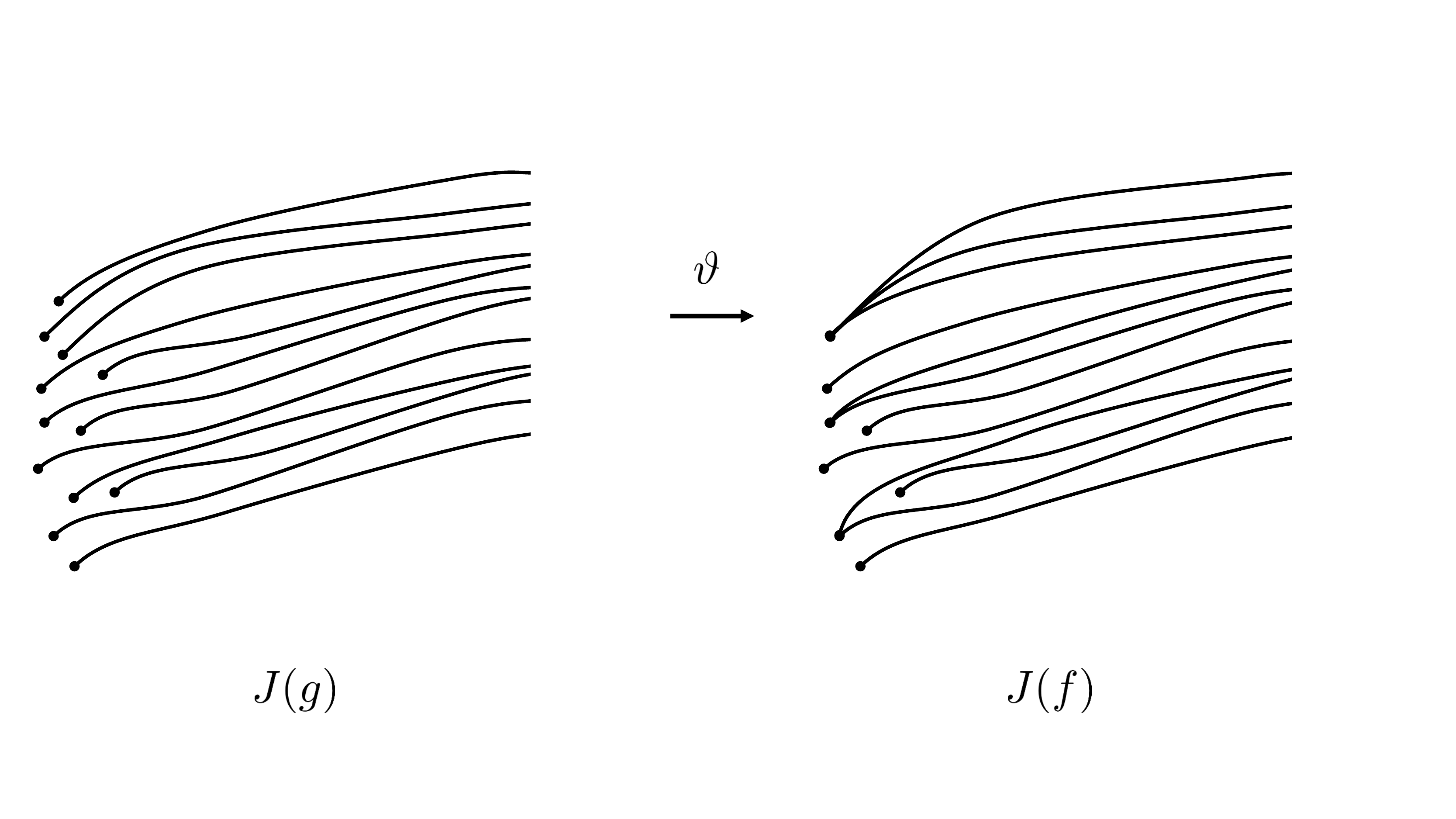}
  \caption{An illustration of the Julia sets in Theorem~\ref{theo:pinchedcantorbouquet}. The function $g$ is of disjoint type and finite order, and so $J(g)$ is a Cantor bouquet. The function $f$ is hyperbolic and of finite order, and the function $\vartheta$ is a surjection from $J(g)$ to $J(f)$ which ``pinches'' some endpoints to form a pinched Cantor bouquet.}\label{figcan}
\end{figure}

In a sense, then, Theorem~\ref{theo:conjugacy} is telling us that the Julia continua of a hyperbolic function are obtained by taking the  Julia continua of a disjoint type function -- discussed earlier -- and then ``pinching'' them together at the points which are not in the escaping set. In particular, we can deduce the following, which characterises the topology of the Julia set of a hyperbolic function of finite order. Here a \emph{pinched Cantor bouquet} is a subset of $\C$ that is ambiently homeomorphic to the quotient of a straight brush by a closed equivalence relation on its endpoints; an equivalence relation on a subset $X$ of $\C$ is \emph{closed} if, for each closed subset $A$ of $X$, the union of the equivalence classes of points in $A$ is closed. See Figure~\ref{figcan}.

\begin{theorem}
\label{theo:pinchedcantorbouquet}
If a {\tef} $f$ is hyperbolic and finite order, then $J(f)$ is a pinched Cantor bouquet.
\end{theorem}
\begin{proof}
In fact this result is an almost immediate consequence of our previous results. First we use Theorem~\ref{theo:conjugacy} to obtain a disjoint type finite order function $g$, whose Julia set is related to that of $f$ by the continuous function $\vartheta$. 

We then note from Theorem~\ref{theo:finiteorder} and Theorem~\ref{theo:cantorbouquet} that the Julia set of $g$ is a Cantor bouquet, and that the only points of $J(g)$ which are not also in $I(g)$ are endpoints. The result then follows.
\end{proof}
\section{A brief overview of other subclasses of $\B$}
\label{s.others}
For reasons of brevity, we have only been able to discuss a few of the subclasses of class $\B$ that have been studied.
In fact, there are many other ways to classify the behaviour of functions in class $\B$, generally depending on the properties of the singular and postsingular sets. The following brief list gives some examples.
\begin{enumerate}
\item A function $f \in \B$ is called \emph{subhyperbolic} if $P(f) \cap F(f)$ is compact and $P(f) \cap J(f)$ is finite. (In other words, we relax the definition of hyperbolic functions to allow finitely many points of $P(f)$ to lie in the Julia set.) These functions were studied in \cite{MR2912445}. It is shown there, that -- with the additional assumptions that $J(f) \cap AV(f) = \emptyset$ and that the local degree of $f$ at points of $J(f)$ is uniformly bounded -- then $J(f)$ is a pinched Cantor bouquet, thereby generalising Theorem~\ref{theo:pinchedcantorbouquet}. An example of such a function is $f(z) = \pi \sinh z$. This function also has the property that $J(f) = \C$.
\item A function $f \in \B$ is called \emph{geometrically finite} if $S(f) \cap F(f)$ is compact and $P(f) \cap J(f)$ is finite. (In other words, we relax the definition of a subhyperbolic function to allow $P(f) \cap F(f)$ no longer to be bounded). For these maps, it can be shown that the Fatou set is either empty or consists of finitely many attracting or parabolic basins. Such maps were studied in \cite{MR2650792}.
\item A function $f \in \B$ is called \emph{postsingularly bounded} if $P(f)$ is bounded. It can be readily shown that all geometrically finite maps are postsingularly bounded. These maps were studied in \cite{MR2346947}, where it was shown that if $f \in \B$ is postsingularly bounded, then Eremenko's conjecture holds for $f$. A more recent, and much more detailed, study of this class of maps was given very recently in \cite{AnnaLasse}. The authors prove a generalisation of the famous Douady-Hubbard landing theorem for postsingularly bounded entire functions.
\end{enumerate}

We illustrate the inclusions for the various classes of functions in Figure~\ref{figinc}. All the inclusions are strict; see \cite{HMBthesis} for examples of functions that illustrate this.

\begin{figure}
	\includegraphics[width=14cm,height=9cm]{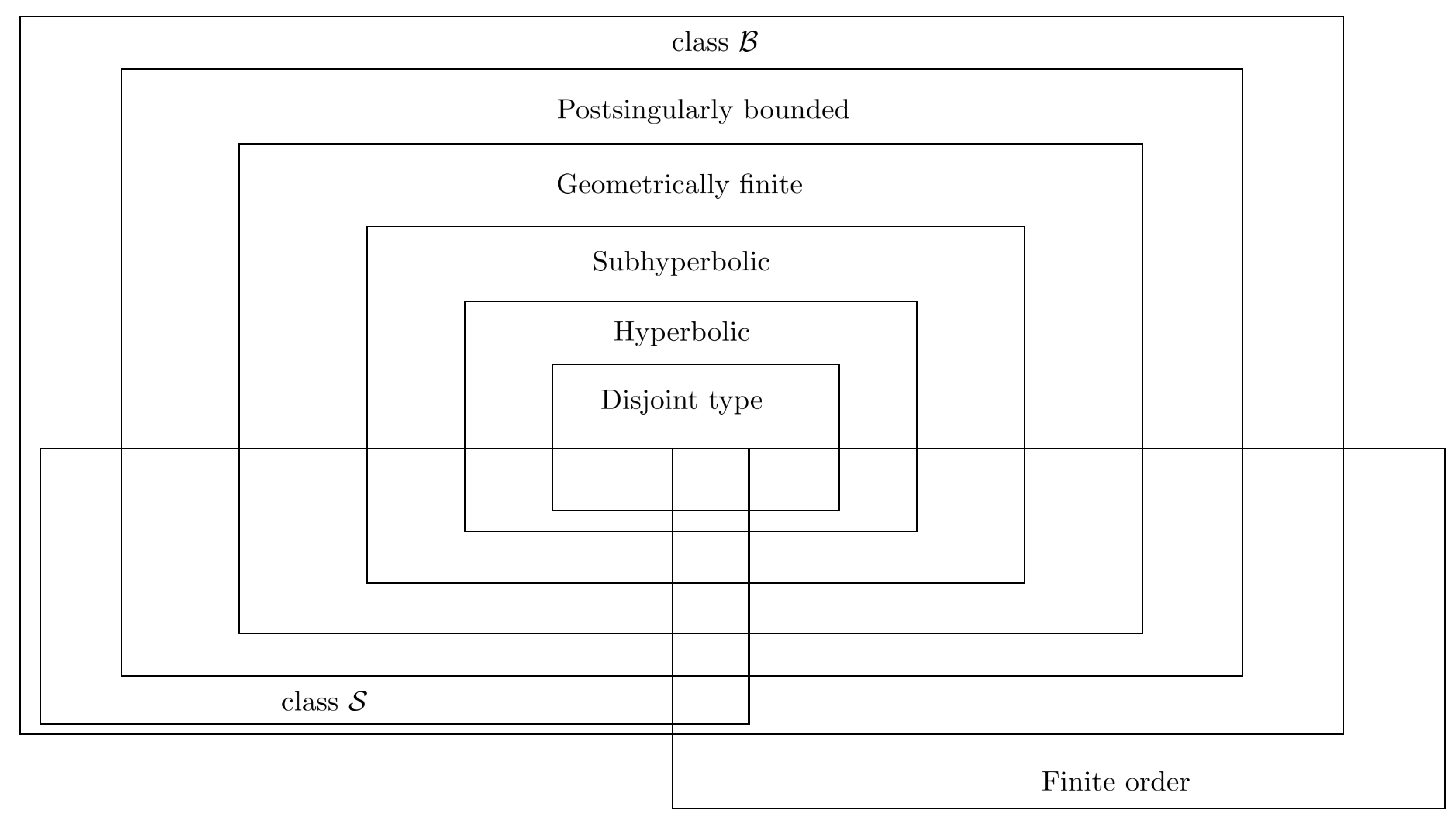}
  \caption{An illustration of relationships between some of the classes of functions discussed.}\label{figinc}
\end{figure}

\section{Constructing functions in class $\mathcal{S}$ and class $\mathcal{B}$}
\label{s.bishop}
\subsection{Techniques for constructing {\tef}s}
Many different techniques have been used to construct {\tef}s with novel dynamical properties. Baker \cite{MR0153842} was the first to use infinite products. It seems, though, that most examples constructed in this way have an unbounded set of critical values, and so lie outside the class $\B$. (An example of a {\tef} in the class $\B$ defined using an infinite product was given in \cite{2016arXiv160804600C}.)

Eremenko and Lyubich \cite{MR918638} pioneered the use of approximation theory. However, in general, it does not seem possible to ensure that the resulting functions lie in the Eremenko-Lyubich class, as this technique gives insufficient control on the set of singular values.

The use of Cauchy integrals to create {\tef}s with novel dynamical properties seems to originate with Stallard \cite{MR1145621, MR1458228}, and also \cite{MR1760674} mentioned earlier. In the first of those papers Stallard studied functions, given in \cite{MR0344042}, that are of the form
\[
E_K(z) = \frac{1}{2\pi i} \int_L \frac{\exp(e^t)}{t - z} \ dt - K, \qfor K > 0,
\]
where the integral is taken around a curve $L$, which is the boundary of the region
\[
\{ x + iy \in \C : x > 0, \ |y| < \pi \}.
\]
The examples constructed by Stallard are all in class $\B$. This technique was also used in \cite{MR2753600} to construct the {\tef} mentioned at the start of Section~\ref{s.finorder} for which the Julia set contains no non-trivial path-connected components. Rempe-Gillen significantly generalised this technique in \cite{MR3214678}.

We also note, in passing, a technique due to MacLane and Vinberg, which can be used to construct {\tef}s with a pre-assigned sequence of real critical values. This technique was used in \cite{MR3433280} to construct hyperbolic functions with certain dynamical properties.

Our main goal in this section, however, is to discuss two powerful and related techniques recently introduced by Bishop. The first is the simplest. The second is more powerful and flexible, but also somewhat more complicated.

Since Bishop makes extensive use of both quasiconformal and quasiregular maps, we first need to discuss these functions, although we will avoid technicalities. Very roughly speaking a quasiconformal map is a homeomorphism that is ``almost'' conformal, in the sense that it maps infinitesimal circles to infinitesimal ellipses, with a uniform bound on the eccentricity of the ellipses. All conformal maps are quasiconformal, but the class of quasiconformal maps is much larger and less rigid. The class of quasiregular maps is obtained from the class of quasiconformal maps by relaxing the requirement of injectivity. Hence quasiregular maps provide, in a sense, a generalisation of the class of analytic maps. Both quasiconformal and quasiregular maps have been used frequently in the study of complex dynamics; see \cite{MR3445628} for more information. Finally, we refer to the monographs \cite{MR1238941} and \cite{MR950174} for a more detailed and precise treatment of quasiconformal and quasiregular maps. 

\subsection{Bishop's ``simpler'' construction.}
The construction discussed in this subsection is from \cite{MR3384512}, which itself was inspired by \cite{MR3214678}. First we define a \emph{model}. Let $\mathbb{H}$ denote the right half-plane $\mathbb{H} := \{ z  : \operatorname{Re}(z) > 0\}$. Suppose that $I$ is an index set which is at most countably infinite, and that $$\Omega := \bigcup_{j \in I} \Omega_j,$$ is a disjoint union of unbounded simply connected domains. Suppose also that, for each $j\in I$, there is a conformal map $\tau_j : \Omega_j \to \mathbb{H}$. Let $\tau$ be the map $\tau : \Omega \to \C$ which is equal to $\tau_j$ on $\Omega_j$. Suppose that the following conditions are satisfied:
\begin{enumerate}[(i)]
\item Sequences of components of $\Omega$ accumulate only at infinity.
\item The boundary of $\Omega_j$ is connected, for $j \in I$.
\item If $(z_n)_{n\in\N}$ is a sequence of points of $\Omega$ such that $\tau(z_n)\rightarrow\infty$ as $n\rightarrow\infty$, then $z_n\rightarrow \infty$ as $n\rightarrow\infty$.
\end{enumerate}
Finally, set $F := \exp \circ \tau$, so that $F$ maps each $\Omega_j$ conformally to $\C\setminus\overline{\D}$. Then the pair $(\Omega, F)$ is called a \emph{model}. 

Roughly speaking \cite[Theorem 1.1]{MR3384512} says the following. If $(\Omega, F)$ is a model, then there exists a function $f\in\mathcal{B}$ which, in a sense that can be made precise, is ``close'' to $F$. This means that we can construct a {\tef} in class $\B$ with certain properties simply by specifying a model, which is essentially a simple geometric object.

Rather than discuss this in further detail, we will instead highlight the following implication of \cite[Theorem 1.1]{MR3384512}. This result, which is \cite[Theorem 2.5]{lassearclike} (see also \cite[Theorem 1.2]{MR3384512}), follows after letting a function $G \in \Blog$ be a model in Bishop's sense.
\begin{theorem}
\label{theo:bish1}
Suppose that $G \in \Blog$ is of disjoint type, and let g be defined by $g(\exp(z)) = \exp(G(z))$. Then there is a disjoint-type function $f \in \B$ and a quasiconformal map $\phi$, which maps a neighbourhood $U$ of $J(f)$ to a neighbourhood of $\exp J(G)$, such that $\phi \circ f = g \circ \phi$ on $U$.
\end{theorem}
It follows from the conclusion of Theorem~\ref{theo:bish1} that $J(f) = \phi^{-1}(\exp J(G))$. This means that any property of $J(G)$ that is preserved by quasiconformal maps -- such as being path-connected -- also holds for $J(f)$.

Theorem~\ref{theo:bish1} is used extensively by Rempe-Gillen in \cite{lassearclike}. Since the examples in \cite{lassearclike} are somewhat complicated, we refer to that paper for further details.

In general, the functions obtained from Theorem~\ref{theo:bish1} are not in the class $\S$. Bishop gives a similar, but modified, version of this construction in \cite{Bish4}. The same definition of a model is used, but the function generated is always in the class $\S$. However, this comes with some associated loss of control; for example, $f$ may have additional tracts that are not present in the original model.
\subsection{Bishop's more sophisticated construction.}
\label{subs:bishop}
In \cite{Bish3} Bishop introduced a technique called \emph{quasiconformal folding}.  The technique starts with an infinite connected graph that satisfies certain (not particularly restrictive) geometric conditions. Bishop shows how to combine certain quasiconformal maps on the complementary components of the graph into a map continuous across the graph and quasiregular on the whole plane. The existence of an entire function $f$ with similar properties to the quasiregular map follows through the measurable Riemann mapping theorem.

The key result is \cite[Theorem 7.2]{Bish3}, of which we omit some detail. In this result the complex plane is divided by a graph into domains known as \emph{R-components}, \emph{L-components} and \emph{D-components}, with certain quasiconformal maps defined in each. Subject to certain technical constraints, for which we refer to \cite{Bish3}, these components and quasiconformal maps are as follows:
\begin{enumerate}[(1)]
\item All R-components are unbounded. The quasiconformal map on an R-component is the composition of a quasiconformal map to the right half-plane and another map, which can be taken to be $z \mapsto \cosh(z)$.

\item All L-components are also unbounded, and share boundaries only with R-components. The required quasiconformal map on an L-component is the composition of a quasiconformal map to the left half-plane, the exponential map to $\mathbb{D}\backslash\{0\}$ and (if required) a quasiconformal map from $\mathbb{D}$ to $\mathbb{D}$ taking the origin to another point in $\mathbb{D}$.

\item All D-components are bounded, and share boundaries only with R-components. The required quasiconformal map on a D-component is the composition of a quasiconformal map to $\mathbb{D}$, a power map $z \to z^d$ and (if required) a quasiconformal map from $\mathbb{D}$ to $\mathbb{D}$ taking the origin to another point in $\mathbb{D}$.
\end{enumerate}

\begin{figure}
	\includegraphics[width=12cm,height=7cm]{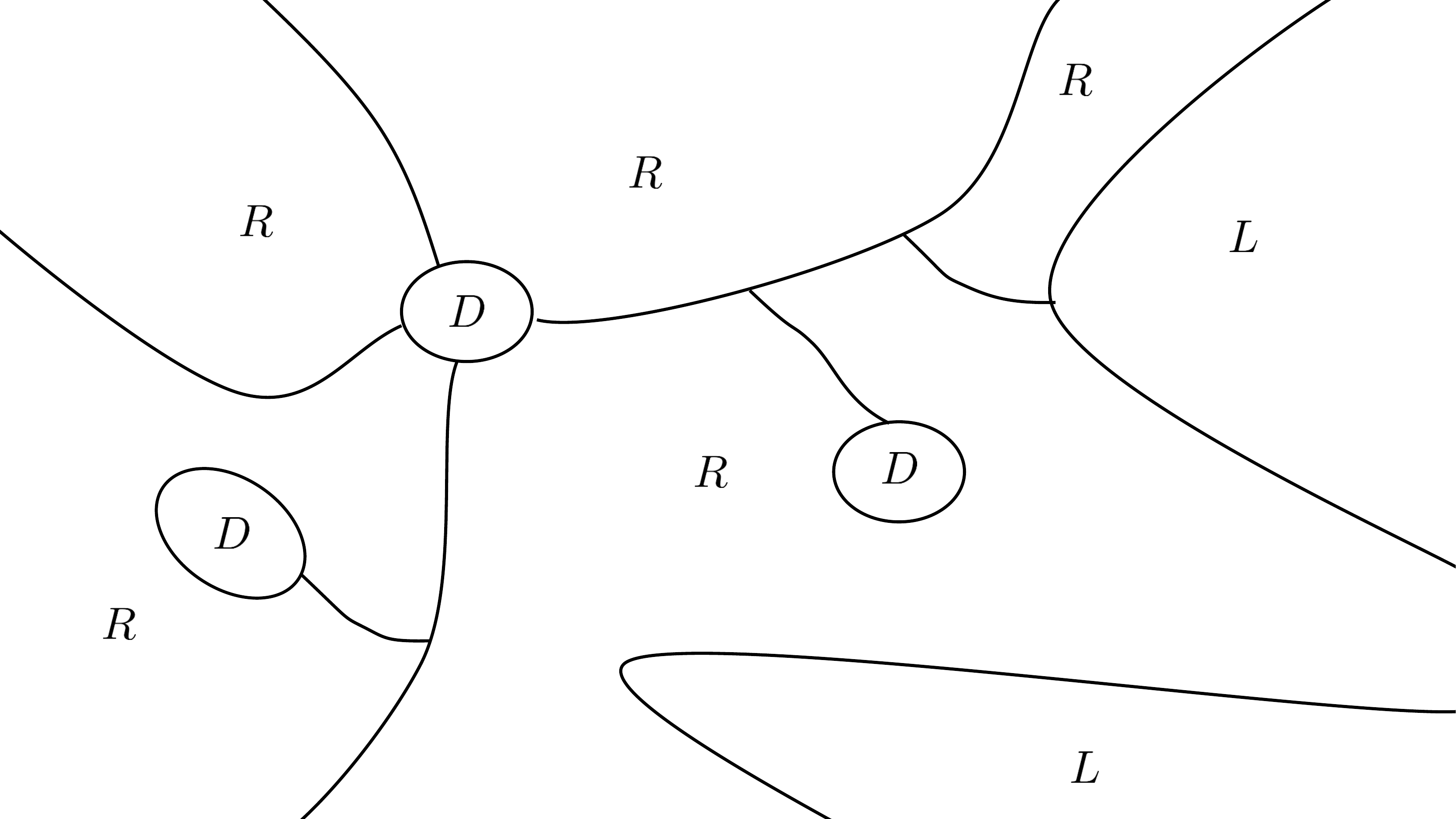}
  \caption{An illustration of a graph showing D, R and L components.}\label{figDRL}
\end{figure}

Very roughly speaking, R-components give rise to exactly two singular values, which are critical values, of bounded degree, and equal to $\pm 1.$  Each D-component gives rise to an additional critical point, which can be of any degree, and an associated critical value lying in $\D$. Each L-component leads to the addition of a finite asymptotic value, which also lies in $\D$. It follows that the function $f$ always lies in the class $\B$ (since $S(f) \subset \overline{\D}$), and may even lie in the class $\S$ if appropriate choices are made on the D- and L-components.

Bishop used his result to construct several class $\B$ and class $\S$ functions with novel properties. We discuss the following specific example, which is particularly noteworthy since it was shown in \cite{MR1196102, MR857196} that no function in class $\S$ has wandering domains.  
\begin{theorem}
\label{theo:bish2}
There is a {\tef} $f \in \B$ such that $f$ has wandering domains.
\end{theorem}
\begin{proof}
We outline the proof. The construction is delicate, in that the properties of the D-components depend on the function resulting from \cite[Theorem 7.2]{Bish3}. We do not attempt to discuss this detail.

The graph used is symmetrical about the real and imaginary axes, and does not use L-components. One R-component is the strip $$S_+ := \{ z = x + iy : x > 0, |y| < \pi / 2 \}.$$ The quasiconformal map in $S_+$ is the map $z \mapsto \cosh(\lambda\sinh(z))$, where $\lambda \in \pi \N$ is chosen sufficiently large so that the point $\frac{1}{2}$ iterates to infinity along the real axis	.

\begin{figure}[ht]
	\centering
	\includegraphics[width=12cm,height=9cm]{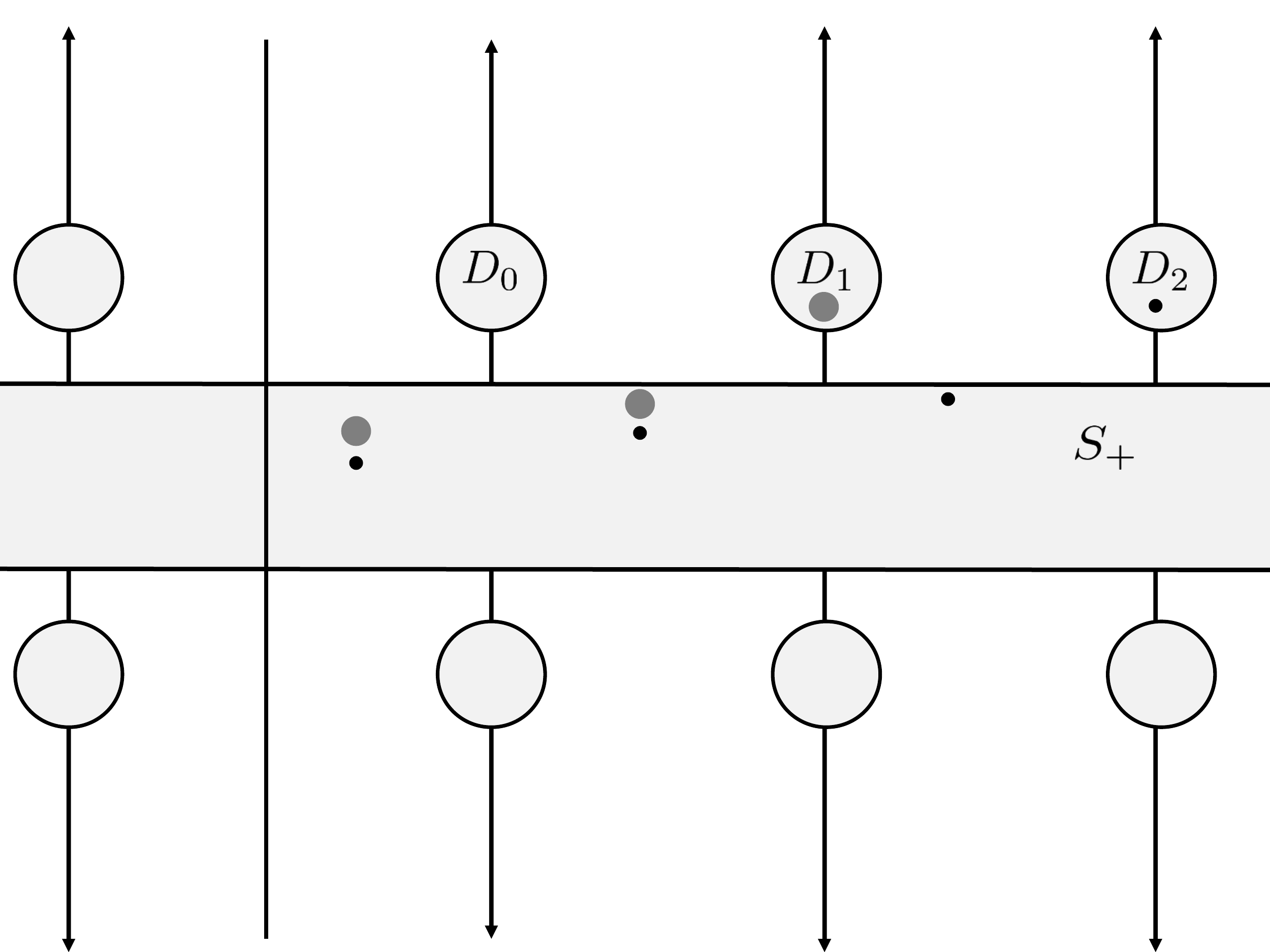}
	\caption{Sketch of the graph for the proof of Theorem~\ref{theo:bish2}. Three consecutive images of the disc $D_0$ are shown in dark gray, and four consecutive images of the disc $D_1$ are shown in black.}
  \label{fig41}
\end{figure}

The D-components are disjoint discs of unit radius, centred at points of imaginary part $\pm \pi$. The quasiconformal maps on these D-components are compositions of a translation to $\mathbb{D}$, a power map of high degree and a quasiconformal map that takes the origin to a point close to $\frac{1}{2}$. The positioning of the D-components, the degree of the power map and the choice of the point close to $\frac{1}{2}$ are all carefully controlled.

The remainder of the complex plane is divided into R-components, but since the dynamics in these components does not affect the example, the quasiconformal maps are not specified.

Choosing a small domain $U$ in $S_+$ close to some point with real part $\frac{1}{2}$ and with positive imaginary part, it is shown that the iterates of $ U $ under $ f $ follow the orbit of $\frac{1}{2}$ until -- through careful choice of the location of the D-components -- the $n$th iterate (say) lands in a D-component. The quasiconformal map in this D-component is selected so as to reduce the diameter of $f^n(U)$ by a large factor (by using a power map of sufficiently high degree), and return it even closer to $\frac{1}{2}$. Subsequent iterates again follow the orbit of~$\frac{1}{2}$ but, because they start closer to this point, they do so for longer before landing in a D-component further from the origin.  Bishop shows that $U \subset F(f)$. It follows that $ U $ is a wandering domain since the iterates in $U$ have both bounded and unbounded sub-orbits, and this behaviour is impossible in a periodic or pre-periodic domain.
\end{proof}
The construction in Theorem~\ref{theo:bish2} was recently modified in \cite{FJL} to give a {\tef} $f \in \B$ with wandering domains in which $f$ is univalent. \\

%
%
%
\emph{Acknowledgment:} The author is very grateful to Simon Albrecht, Vasiliki Evdoridou and John Osborne for their careful readings of early drafts of this survey, and also to a number of others who gave valuable feedback. 
%
%
%
\bibliographystyle{alpha}
\bibliography{../../Research.References}
\end{document}